\documentclass[12pt]{amsart}
\usepackage{amsmath,mathrsfs}

\usepackage[ps2pdf,colorlinks=true,urlcolor=blue,
citecolor=red,linkcolor=blue,linktocpage,pdfpagelabels,bookmarksnumbered,bookmarksopen]{hyperref}
\usepackage{epsfig,graphicx,color,mathrsfs}
\usepackage[english]{babel}

\usepackage[left=2.8cm,right=2.8cm,top=2.7cm,bottom=2.7cm]{geometry}

\numberwithin{equation}{section}

\newcommand{\N}{{\mathbb N}}

\newcommand{\R}{{\mathbb R}}
\newcommand{\Dis}{{\mathbb D}}
\newcommand{\Sf}{{\mathbb S}}

\newcommand{\eps}{\varepsilon}

\newcommand{\dom}[1]{{\rm dom}(#1)}
\newcommand{\ws}[1]{|d#1|}

\newcommand{\gra}[1]{{\mathcal G}_{#1}}

\newcommand{\epi}[1]{{\rm epi}\left(#1\right)}

\newcommand{\refe}[1]{{(\ref{#1})}}
\newcommand{\dys}{\displaystyle}

\renewcommand{\theta}{\vartheta}
\newcommand{\fix}{{{\rm Fix}(G)}}
\newcommand{\fixs}{{{\rm Fix}(G_N)}}

\numberwithin{equation}{section}
\newtheorem{theorem}{Theorem}[section]
\newtheorem{proposition}[theorem]{Proposition}
\newtheorem{lemma}[theorem]{Lemma}
\newtheorem{remark}[theorem]{Remark}
\newtheorem{corollary}[theorem]{Corollary}
\newtheorem{definition}[theorem]{Definition}
\theoremstyle{definition}

\title{On the Palais principle for non-smooth functionals}

\author{Marco Squassina}
\address{Dipartimento di Informatica
\newline\indent
Universit\`a degli Studi di Verona
\newline\indent
C\'a Vignal 2, Strada Le Grazie 15, I-37134 Verona, Italy}
\email{marco.squassina@univr.it}

\thanks{Research partially supported by PRIN: {\em Metodi Variazionali e Topologici
nello Studio di Fenomeni non Lineari}}

\begin{document}
	
\date{\today}



\begin{abstract}
If $G$ is a compact Lie group acting linearly on a Banach space $X$ and
$f$ is a $G$-invariant function on $X$, we provide some new versions of the so-called Palais' principle of symmetric criticality 
for $f:X\to\bar\R$, in the framework of non-smooth critical point theory. We apply the results to a class
of quasi-linear PDEs associated with invariant functionals which are merely lower semi-continuous and thus could
not be treated by previous non-smooth versions of the principle in the literature.
\end{abstract}
\maketitle


\bigskip
\begin{center}
\begin{minipage}{11cm}
\footnotesize
\tableofcontents
\end{minipage}
\end{center}

\bigskip

\section{Introduction}
Let $G$ be a compact Lie group which acts linearly on a Banach space $X$, let $f:X\to\bar\R$ be a
$G$-invariant functional and consider the fixed point set of $X$ under $G$,
$$
\fix=\bigcap_{g\in G}\{u\in X: (g-{\rm Id})u=0\},
$$
which is a closed subspace of $X$.
In order to detect critical points of $f$, a possible
approach is based upon what is currently known as Palais' principle of symmetric criticality, that is
looking for critical points of $f|_\fix$ in order to locate 
critical points of $f$. The origin of the principle is rather unclear and
the first implicit use seems to trace back to 1950 inside {\sc H. Weyl}
derivation of the Einstein field equations \cite{weyl}. Next {\sc S. Coleman} \cite{coleman}
makes an explicit reference to it around 1975. 
For $G$-invariant functionals of class $C^1$ the principle was rigorously formulated by
{\sc R.S. Palais} in 1979 in a celebrated paper~\cite{palaiscmp} (see also~\cite{ladiz,palais1,palais2}), reading as 
\begin{equation}
	\label{formul0}
\begin{cases}
u\in \fix & \\
df|_\fix(u)=0 &
\end{cases}
\,\,\,\Longrightarrow\quad  df(u)=0.
\end{equation}
In other words, in order for an invariant point $u$ to be critical for $f$, it is sufficient that
it be critical for $f|_\fix$, namely if the directional derivative $f'(u)v$ vanishes along any direction $v$ tangent to $\fix$,
then it must vanish along the directions transverse to $\fix$ as well. 
Although the implication is valid in a rather broad context, 
there are of course some pathological counterexamples showing that it can fail to hold (see~\cite[Section 3]{palaiscmp}). 

The validity of Palais' principle of symmetric criticality has a powerful impact on applications to nonlinear problems of 
mathematical physics which are set on unbounded domains (for instance $\R^N$, half-spaces or strips),
are invariant under some linear transformations, such as rotations, thus exhibiting 
symmetry (for instance, spherical or cylindrical) and are associated with a suitable 
energy functional $f$. In fact, in this framework, it
is often the case that the solutions of a PDE
live in a Banach space $X$ which is continuously embedded into a space Banach $W$
but, unfortunately, the injection $X\hookrightarrow W$ fails to be compact. Typically
the unboundedness of the domain where the problem is set prevents these inclusions to be compact and 
the problem cannot be solved by standard methods of nonlinear analysis.
On the contrary, in many interesting concrete situations,
for suitable compact Lie groups $G$, the compactness of the embedding 
${\rm Fix}(G)\hookrightarrow W$ is restored. We refer the reader to
the classical works by {\sc P.L. Lions} \cite{lionscomp} and {\sc W.A. Strauss} \cite{strausscomp} (see for instance
\cite[Section 1.5]{willem} for a list of the results of \cite{lionscomp,strausscomp}).
This fact offers a fruitful tool in order to find a critical point $u$ of $f|_\fix$ at the
Mountain Pass level (hence a critical point of $f$ by the principle and, in turn, a solution to the associated
problem, in a suitable sense). In the language of non-smooth critical point theory the implication~\eqref{formul0}
corresponds to the property
\begin{equation}
	\label{formul1}
\begin{cases}
u\in \fix & \\
0\in\partial f|_\fix(u) &
\end{cases}
\,\,\,\Longrightarrow\quad  0\in\partial f(u),
\end{equation}
where $\partial f(u)$ is the subdifferential of $f$ at $u$. In the current literature,
some papers have already investigated the validity of this implication, for locally Lipschitz functions 
by {\sc W.~Krawcewicz} and {\sc W.~Marzantowicz} \cite{marza}, for
sums of $C^1$ (resp. locally Lipschitz) functions with lower semi-continuous convex functionals 
in the nice work by {\sc J. Kobayashi} and {\sc M. Otani} \cite{koba}
(resp.~by {\sc A.~Kristaly, C.~Varga} and {\sc V.~Varga} \cite{krisvarga}). 
On the other hand, a suitable definition of $\partial f(u)$ for any function $f:X\to\bar\R$ has been recently
given by {\sc I. Campa} and {\sc M. Degiovanni} in \cite{campad}, consistently
with the subdifferential of convex functions and containing the subdifferential in the sense of 
{\sc F.H.~Clarke} \cite{clarke} for locally Lipschitz functions. 

The main goal of this paper is to obtain
two new Palais' symmetric criticality type principles for non-smooth functionals and furnish an application to elliptic PDEs
associated with merely lower semi-continuous functionals, thus out of the frame of the previous literature
available on the subject.
In Section~\ref{recallsns} we will recall some definitions and tools
of non-smooth analysis. In Section~\ref{firstab} we provide a first non-smooth version of the 
symmetric criticality principle, Theorem~\ref{abstr11}, for a class of functionals 
which includes, in particular, the cases covered by the previous literature (cf.~Proposition~\ref{consistenza}). 
Theorem~\ref{abstr11} asserts that, under suitable assumptions, property~\eqref{formul1} holds true. 
In Section~\ref{secondab} we shall obtain a second abstract result in a very general framework 
(Theorem~\ref{abstr22}). The result proves, under suitable assumptions, the implication
\begin{equation}
	\label{formul2}
\begin{cases}
u\in \fix & \\
0\in\partial f|_\fix(u) & \\
\text{$\exists$ $V_u\subset X$ invariant, dense subspace}
\end{cases}
\Longrightarrow\quad  \forall v\in V_u:\,\, f'(u)v=0.
\end{equation}
In fact, in many concrete situations, the classical directional derivative exists on suitable dense subspaces
of $X$, and the criticality of $u$ over these spaces is often sufficient to guarantee that, actually, $u$ is the solution
of a corresponding PDE, in a generalized or distributional sense (cf.~\cite{pelsqu}). In Section~\ref{concrete}, we shall discuss 
an application (Theorem~\ref{appl}) of Theorem~\ref{abstr22} showing that, for an unbounded domain 
$\Omega\subset\R^N$ invariant under a group $G_N$ of rotations, under suitable assumptions
there exists a nontrivial $G_N$-invariant solution to the quasi-linear problem	
\begin{equation*}
\begin{cases}
	\displaystyle
\int_\Omega j_t(u,|Du|)\frac{Du}{|Du|}\cdot D v +
	\int_\Omega j_s(u,|Du|)v+\int_\Omega |u|^{p-2}uv=\int_\Omega |u|^{q-2}uv, & \\
\,\,\forall v\in C^\infty_c(\Omega). &
\end{cases}
\end{equation*}
As a consequence (Corollary~\ref{appl-cor1}), in the case $\Omega=\R^N$, we will prove that this problem admits
a radial solution, extending \cite[Theorems 1.29]{willem} to the quasi-linear case.
We would like to stress that the previous non-smooth 
versions of Palais' symmetric criticality principle in the literature \cite{koba,krisvarga} cannot be used in order to prove
the existence of solutions to the above elliptic problem, since its associated invariant functional 
is no more than lower semi-continuous (and non convex). For an overview on
these classes of quasi-linear problems, we refer the interested reader to the monograph~\cite{sq-monog}.
\vskip2pt
\noindent
In contrast with the use of Palais's symmetric criticality principle discussed above,
a direct approach is also possible, where, in the search of symmetric solutions, one avoids restricting the functional to $\fix$.
This approach has been recently developed by the author \cite{sq-ccm} for a class of lower semi-continuous functionals, extending 
previous results from \cite{jvsh1} valid for $C^1$ functionals. 
In Section~\ref{concretesec} we state and prove a new abstract result, Theorem~\ref{main}, yielding
the existence of a Palais-Smale sequence $(u_h)\subset X$ for $f$ at level $c$ which is compact in a 
suitable Banach space $W$ containing $X$, and it becomes more and more symmetric, as $h$ increases. 
A final remark is now adequate. Let $\Dis$ and $\Sf$ are the closed unit ball and the sphere in $\R^m$ ($m\geq 1$) 
respectively and $\Gamma_0\subset C(\Sf,X)$, consider the (unrestricted) Mountain Pass energy level
$$
c:=\inf_{\gamma\in\Gamma}\sup_{\tau\in \Dis} f(\gamma(\tau)),\,\,\,\,\quad
\Gamma:=\big\{\gamma\in C(\Dis,X):\gamma|_{\Sf}\in \Gamma_0\big\}.
$$
Then, setting 
$$
\Gamma_0^\fix:=\{\gamma\in\Gamma_0: \gamma(\Sf)\subset \fix \},
$$ 
one can consider, for $f|_{\fix}:\fix\to\bar\R$, the restricted Mountain Pass energy level
\begin{equation}
	\label{critinv}
c^\fix :=\inf_{\gamma\in\Gamma^\fix}\sup_{\tau\in \Dis} f(\gamma(\tau)),\,\,\,\,\quad
\Gamma^\fix :=\{\gamma\in C(\Dis,\fix):\gamma|_{\Sf}\in \Gamma_0^\fix\}.
\end{equation}
Then, automatically, it holds $c\leq c^\fix$ in light of
$\Gamma^\fix\subset\Gamma$. Hence, in general, the strategy based upon Palais' criticality principle 
ensures a simpler approach to concrete problems but, as a drawback, 
the minimality property of the Mountain Pass energy level $c$ might be lost,
which is not the case in the above described direct approach.

\section{Some notions from non-smooth analysis}
\label{recallsns}
In this section we consider abstract notions
and results that will be used in the proof 
of the main result. For the definitions, we refer e.g.~to~\cite{campad,cdm,dm}.
Let $X$ be a metric space, $B(u,\delta)$ the open ball of center
$u$ and of radius $\delta$ and let $f:X\to\bar\R$
be a function. We set
\begin{equation*}
\dom{f} = \left\{u\in X:\,f(u)<+\infty\right\}
\quad\text{and}\quad
\epi{f}=\left\{(u,\xi)\in X\times\R:\,f(u)\leq\xi\right\}.
\end{equation*}

We recall the definition of the weak slope.

\begin{definition}\label{defslope}
For every $u\in X$ with $f(u)\in\R$, we denote by $|df|(u)$ the supremum
of $\sigma$'s in $[0,\infty)$ such that there exist $\delta>0$ and a continuous map
$$
{\mathcal H}:\,B((u,f(u)), \delta)\cap {\rm epi}(f) \times[ 0, \delta]  \to X,
$$
satisfying
\begin{equation*}
d({\mathcal H}((\xi,\mu),t),\xi) \leq t,\qquad
f({\mathcal H}((\xi,\mu),t)) \leq f(\xi)-\sigma t.
\end{equation*}
for all $(\xi,\mu)\in B((u,f(u)), \delta)\cap {\rm epi}(f)$
and $t\in [0,\delta]$. The extended real number $|df|(u)$ is called the weak
slope of $f$ at $u$.
\end{definition}

\begin{remark}\rm
Let $X$ be a metric space, $f:X \to \R $ a continuous function,
and $u\in X$. Then $|df|(u)$ is the supremum of the real numbers $ \sigma$ in
$[0,\infty)$ such that there exist $ \delta >0$
and a continuous map
$
{\mathcal H}\,:\,B(u, \delta) \times[ 0, \delta]  \to X,
$
such that, for every $v$ in $B(u,\delta) $, and for every
$t$ in $[0,\delta]$ it results
\begin{equation*}
d({\mathcal H}(v,t),v) \leq t,\qquad
f({\mathcal H}(v,t)) \leq f(v)-\sigma t.
\end{equation*}
\noindent
If furthermore $X$ is an open subset of a normed space $E$ and $f$ is a function of class $C^1$
on $X$, then $|df|(u)=\|df(u)\|$, for every $u\in X$ (see~\cite[Corollary 2.12]{dm}).
\end{remark}

\noindent
Let us define the function $\gra{f}:\epi{f}\to\R$ by 
\begin{equation}
\label{defg}
\gra{f}(u,\xi)=\xi.
\end{equation}
In the following, $\epi{f}$ will be endowed with the metric
$$
d((u,\xi),(v,\mu))=
(d(u,v)^2+(\xi-\mu)^2)^{1/2},
$$
so that the function $\gra{f}$ is Lipschitz continuous of constant $1$.
\vskip3pt
\noindent
We have the following

\begin{proposition}
\label{Prodefnwslsc}
For every $u\in X$ such that $f(u)\in\R$, we have
\begin{equation*}
\ws{f}(u)=
\begin{cases}
\dys \frac{\ws{\gra{f}}(u,f(u))}{{\sqrt{1-\ws{\gra{f}}(u,f(u))^2}}},
& \text{if $\ws{\gra{f}}(u,f(u))<1$}, \\
\noalign{\vskip4pt}
+\infty, & \text{if $\ws{\gra{f}}(u,f(u))=1$}.
\end{cases}
\end{equation*}
\end{proposition}
\noindent
In order to apply the abstract theory
to the study of lower semi-continuous functions, the following condition is crucial
\begin{equation}
\label{keycond}
\forall (u,\xi)\in\epi{f}:\,\,\,
f(u)<\xi\,\,\,\Longrightarrow\,\,\,\ws{\gra{f}}(u,\xi)=1.
\end{equation}
We refer the reader to~\cite{cdm,dm} where this is discussed.

\vskip4pt
\noindent
Let now $X$ and $X^*$ denote a real normed space and its topological dual, respectively. 
We recall~\cite[Definitions 4.3 and 5.5]{campad}, namely the definition of the generalized directional derivative.

\begin{definition}
	\label{effe0}
Let $u\in X$ with $f(u)\in\R$. For every $v\in X$ 
and $\eps>0$ we define $f_\eps^\circ(u;v)$ to be the infimum of the $r$'s in $\R$ such that
there exist $\delta>0$ and a continuous function 
$$
{\mathcal V}: B_\delta (u,f(u))\cap {\rm epi}(f)\times ]0,\delta]\to B_\eps(v),
$$
which satisfies
\begin{equation*}
f(\xi+t{\mathcal V}((\xi,\mu),t))\leq \mu+rt,
\end{equation*}	
whenever $(\xi,\mu)\in B_\delta(u,f(u))\cap {\rm epi}(f)$
and $t\in ]0,\delta]$. Finally, we set
$$
f^\circ(u;v):=\sup_{\eps>0}f_\eps^\circ(u;v)=\lim_{\eps\to 0^+}f_\eps^\circ(u;v).
$$	
We say that $f^\circ (u;v)$ is the directional derivative of $f$ at $u$ with respect to $v$.
\end{definition}

\noindent
It is readily seen that the directional derivative $f^\circ(u;v)$ does not change if the norm of $X$ is substituted by an equivalent one.
\vskip5pt
\noindent
According to~\cite[Definition 4.1, Corollary 4.7]{campad}, we recall the notion of subdifferential.

\begin{definition}
	\label{defsub}
Let $u\in X$ with $f(u)<+\infty$. We set
$$
\partial f(u):=\big\{\alpha\in X^*: \langle\alpha,v\rangle\leq f^\circ (u;v),\,\,\,\forall v\in X\big\}.
$$
The set $\partial f(u)$ is convex and weakly* closed in $X^*$.
\end{definition}

\noindent
Again, the set $\partial f(u)$ does not change if the norm of $X$ is substituted by an equivalent one.
\vskip5pt
\noindent	
We conclude the section by recalling \cite[Corollary 4.13 (ii)-(iii)]{campad}), 
establishing the connection between the weak slope $|df|(u)$ and the subdifferential $\partial f(u)$.	
	\begin{proposition}
		\label{connection}
Let $f:X\to \bar\R$ be a functional and $u\in X$ with $f(u)\in\R$. Assume that $|df|(u)<+\infty$.
Then $\partial f(u)\not=\emptyset$ and
$$
\min\{\|\alpha\|:\alpha\in \partial f(u)\}\leq |df|(u).
$$
In particular, $|df|(u)=0$ implies $0\in \partial f(u)$.
\end{proposition}

The previous notions allow us to give the following
\begin{definition}
	\label{cptdef}
We say that $u\in\dom{f}$ is a critical
point of $f$ if $0\in\partial f(u)$.
Moreover, $c\in\R$ is a critical value of
$f$ if there exists a critical point $u\in\dom{f}$
of $f$ with $f(u)=c$.
\end{definition}

\begin{definition}\label{defps}
Let $c\in\R$. We say that $f$ satisfies
the Palais-Smale condition at level $c$ ($(PS)_c$
in short), if every sequence $(u_n)$ in $\dom{f}$ such that
$0\in\partial f(u_n)$ and $f(u_n) \to c$
admits a subsequence $(u_{n_k})$ converging in $X$.
\end{definition}

\section{A first abstract result}
\label{firstab}

Let $G$ be a compact Lie group acting linearly on a real Banach space $(X,\|\cdot\|)$. By suitably renorming $X$ with an
equivalent norm, we may assume without loss of generality that the action of $G$ over $X$ is isometric, that is
$$
\forall u\in X,\,\,\forall g\in G:\quad \|gu\|=\|u\|.
$$
For the proof, we refer the reader e.g.~to~\cite[Proposition 3.15]{koba}. 
\vskip3pt
\noindent
Let now $\mu$ denote the normalized {\em Haar measure} and define the map $A:X\to X$, 
known as {\em averaging map} or barycenter map, providing the center of 
gravity of the orbit of a $v\in X$, by 
$$
\forall v\in X:\quad Av:=\int_G gv d\mu(g).
$$
The map $A$ enjoys some useful properties. Firstly, $Av=v$ for all $v\in \fix$.
Moreover, it is a continuous linear projection from $X$ onto $\fix$,
by the left invariance of $\mu$. Finally, if $C\subseteq X$ is a closed, invariant
(namely $gC\subseteq C$ for all $g\in G$), convex subset of $X$, then $A(C)\subseteq C$. See 
for instance~\cite[Section 5]{palaiscmp}.

\subsubsection{The statement}
We consider the following assumption on $f$:
\vskip2pt
\noindent
We assume that for all $u,v\in \fix$ there exist $\rho>0$ and $C>0$ with
\begin{equation}
	\begin{cases}
	\label{assAbs}
	\displaystyle\frac{f(\zeta+t(z-Az))-f(\zeta)}{t}\geq -C\|z-Az\|,  & \\
\forall \zeta\in B_\rho(u)\cap \fix,\,\,\,
\forall z\in B_\rho (v),\,\,\,
\forall t\in]0,\rho]\,\,\,\,\text{with}\,\,\, f(\zeta-tAz)\leq f(u)+\rho.
\end{cases}
\end{equation}

\vskip3pt
\noindent
As stated next, condition \eqref{assAbs} includes, in particular, the classes of functions already considered 
in the previous literature~\cite{koba,krisvarga} on the subject. More precisely, we have the following

\begin{proposition}
	\label{consistenza}
	Let $f:X\to\R\cup\{+\infty\}$ be such that $f=f_0+f_1$, where $f_0:X\to\R$ is a locally lipschitz
	function and $f_1:X\to \R\cup\{+\infty\}$ is a convex, proper, lower semi-continuous and
	$G$-invariant function. Then condition~\eqref{assAbs} is satisfied.
\end{proposition}
\begin{proof}
	Condition~\eqref{assAbs} can be checked independently for $f_0$ and $f_1$. The proof for $f_0$ is trivial.
	Let us turn to the proof for $f_1$. For every $\zeta\in \fix$, $z\in X$ and $t\geq 0$, we consider the set
	$$
	C:=\big\{\eta\in X: f_1(\eta)\leq f_1(\zeta+t(z-Az))\big\}.
	$$
	Of course $C$ is convex and closed, since $f_1$ is convex and lower semi-continuous. Moreover,
	$C$ is $G$-invariant too, since $f_1$ is $G$-invariant. Then, $A(C)\subseteq C$. Recalling 
	that $\zeta\in \fix$, we have $\zeta=A(\zeta+t(z-Az))\in C$, namely $f_1(\zeta+t(z-Az))-f_1(\zeta)\geq 0$,
	which yields~\eqref{assAbs}.
\end{proof}
\vskip2pt
\noindent
The main result of the section is the following

\begin{theorem}
	\label{abstr11}
Let $G$ be a compact Lie group acting linearly on $X$, let
$f:X\to \bar\R$ be a $G$-invariant functional satisfying condition~\eqref{assAbs} 
and $u\in \fix$ with $f(u)\in\R$ 
be a critical point of $f|_{\fix}$. Then $u$ is a critical point of $f$.
\end{theorem}

The theorem is obtained in a quite simple fashion via the abstract machinery developed in \cite{campad}
and which does not actually involve any a-priori regularity assumption on the map $f$. Assumption \eqref{assAbs} 
only enters in the proof of Proposition \ref{legamefix}.

\vskip4pt
\subsubsection{Some preliminary results}

Next we state some preparatory results for Theorem~\ref{abstr11}.

\begin{proposition}
	\label{invariance}
	Let $G$ be a compact Lie group acting linearly on $X$, let
	$f:X\to \bar\R$ be a $G$-invariant functional and $u\in \fix$ with $f(u)\in\R$.
	Then
	$$
	\forall v\in X,\,\,\forall g\in G:\quad
	f^\circ(u;v)=f^\circ(u;gv).
	$$
\end{proposition}
\begin{proof}
	As recalled, the action of $G$ over $X$ can be assumed to be isometric.
Given $v\in X$ and $g\in G$, fix $\eps>0$. Let $r\in\R$ such that 
there exist $\delta>0$ and a continuous function 
$$
{\mathcal V}: B_\delta (u,f(u))\cap {\rm epi}(f)\times ]0,\delta]\to B_\eps(v),
$$
according to Definition \ref{effe0}. Since $u\in \fix$ and $f$ is $G$-invariant, 
if $(\xi,\mu)\in B_\delta(u,f(u))\cap {\rm epi}(f)$, then also $(g\xi,\mu)\in B_\delta(u,f(u))\cap {\rm epi}(f)$
as $\|g\xi-u\|=\|g\xi-gu\|=\|\xi-u\|\leq \delta$ and $f(g\xi)=f(\xi)\leq\mu$.
Then, we are allowed to consider the continuous function
$$
\hat {\mathcal V}: B_\delta (u,f(u))\cap {\rm epi}(f)\times ]0,\delta]\to B_{\eps}(g^{-1}v),
$$
defined by setting
$$
\hat {\mathcal V}((\xi,\mu),t):=g^{-1} {\mathcal V}((g\xi,\mu),t)
$$
for all $(\xi,\mu)\in B_\delta(u,f(u))\cap {\rm epi}(f)$
and $t\in ]0,\delta]$. It follows that $\hat {\mathcal V}$ is well defined as
$$
\|\hat {\mathcal V}((\xi,\mu),t)-g^{-1}v\|=\|g^{-1} {\mathcal V}((g\xi,\mu),t)-g^{-1}v\|=\|{\mathcal V}((g\xi,\mu),t)-v\|\leq \eps.
$$
Furthermore, it follows
\begin{align*}
f(\xi+t\hat {\mathcal V}((\xi,\mu),t)) &=f(\xi+tg^{-1} {\mathcal V}((g\xi,\mu),t))  \\
& = f(g\xi+t {\mathcal V}((g\xi,\mu),t))  
 \leq \mu+rt,
\end{align*}
for all $(\xi,\mu)\in B_\delta(u,f(u))\cap {\rm epi}(f)$
and $t\in ]0,\delta]$. By definition, it follows $f_\eps^\circ(u;g^{-1}v)\leq r$. By 
the arbitrariness of $r$, $f_\eps^\circ(u;g^{-1}v)\leq f_\eps^\circ(u;v)$.
In a similar fashion, the opposite inequality follows as well. In conclusion, 
$f_\eps^\circ(u;g^{-1}v)=f_\eps^\circ(u;v)$ for every $\eps>0$. 
The assertion follows by letting $\eps\to 0$.
\end{proof}

We now have the following

\begin{proposition}
	\label{primostep}
	Let $G$ be a compact Lie group acting linearly on $X$, let
	$f:X\to \bar\R$ be a $G$-invariant functional and $u\in \fix$ with $f(u)\in\R$.
	Assume that
	$$
	\forall v\in \fix :\quad
	f^\circ(u;v)\geq 0.
	$$
Then
	$$
	\forall v\in X:\quad f^\circ(u;v)\geq 0.
	$$
\end{proposition}
\begin{proof}
	Let $\eps>0$ and consider the set
	$$
	C_\eps=\{v\in X: f^\circ(u,v)\leq -\eps\}.
	$$
	Assume, by contradiction, that $C_\eps\not=\emptyset$. By \cite[Corollary 4.6]{campad} the map $\{v\mapsto f^\circ(u;v)\}$
	is convex and lower semi-continuous. Then $C_\eps$ is convex and closed. Furthermore $C_\eps$ is $G$-invariant in $X$
	by means of Proposition~\ref{invariance}. In turn, $A(C_\eps)\subset C_\eps\cap\fix$. Since $A(C_\eps)\not=\emptyset$, we deduce that
	$C_\eps\cap\fix\not=\emptyset$, which yields a contradiction. By the arbitrariness of $\eps$,
	$$
	\{v\in X: f^\circ(u,v)<0\}=\bigcup_{\eps>0} C_\eps=\emptyset,
	$$
	which proves the assertion.
\end{proof}

The next assertion is the crucial step in order to link the generalized directional
derivatives of $f|_\fix$ with the generalized directional derivatives of $f$ at a point $u\in\fix$.

\begin{proposition}
	\label{legamefix}
	Let $G$ be a compact Lie group acting linearly on $X$, let
	$f:X\to \bar\R$ be a functional which satisfies \eqref{assAbs}
	and $u\in \fix$ with $f(u)\in\R$. Then
	$$
	\forall v\in \fix :\quad
	f^\circ(u;v)\geq (f|_\fix)^\circ(u;v).
	$$
	In particular, if
	$$
	\forall v\in \fix :\quad
	(f|_\fix)^\circ(u;v)\geq 0,
	$$
    then
	$$
	\forall v\in \fix:\quad f^\circ(u;v)\geq 0.
	$$
\end{proposition}
\begin{proof}
	Let $v\in\fix$ and let $\eps\in (0,\rho)$, where $\rho>0$ is the number 
	appearing in assumption \eqref{assAbs} on $f$.
	Moreover, let $r\in\R$ be such that there exist $\delta>0$ and a continuous function
	$$
	{\mathcal V}: B_\delta (u,f(u))\cap {\rm epi}(f)\times ]0,\delta]\to B_{\eps}(v),
	$$
	according to Definition~\ref{effe0}. Then, we choose a number
\begin{equation}
	\label{sceltadelt}
	0<\delta'<\min\Big\{\delta,\frac{\rho}{1+\eps+\|v\|}\Big\},
\end{equation}
	and we define a map
	$$
	\hat{\mathcal V}: B_{\delta'} (u,f(u))\cap {\rm epi}(f)\cap (\fix\times\R)\times ]0,\delta']\to B_{\eps}(v)\cap \fix,
	$$
by setting
$$
\hat{\mathcal V}((\xi,\mu),t):=A{\mathcal V}((\xi,\mu),t),
$$
for all $(\xi,\mu)\in B_{\delta'}(u,f(u))\cap {\rm epi}(f)\cap (\fix\times\R)$
and $t\in ]0,\delta']$.  Notice that, as $\|A\|\leq 1$,
$$
\|\hat{\mathcal V}((\xi,\mu),t)-v\|=\|A{\mathcal V}((\xi,\mu),t)-Av\|\leq \|A\|\|{\mathcal V}((\xi,\mu),t)-v\|\leq \eps,
$$
for all $(\xi,\mu)\in B_{\delta'}(u,f(u))\cap {\rm epi}(f)\cap (\fix\times\R)$
and $t\in ]0,\delta']$, so that $\hat{\mathcal V}$ is well defined.
We shall use assumption~\eqref{assAbs}, applied with
$$
\zeta:=\xi+t Az,\,\, z:={\mathcal V}((\xi,\mu),t),\quad
(\xi,\mu)\in B_{\delta'}(u,f(u))\cap {\rm epi}(f)\cap \fix\times\R,\,\,
t\in ]0,\delta'].
$$ 
Notice that $\zeta\in\fix$, $f(\zeta-tAz)=f(\xi)\leq \mu\leq f(u)+\delta'< f(u)+\rho$ and 
\begin{align}
	\label{zetnormcontr}
\|z-Az\| &=\|{\mathcal V}((\xi,\mu),t)-A{\mathcal V}((\xi,\mu),t)\| \\
&\leq \|{\mathcal V}((\xi,\mu),t)-v\|+\|Av-A{\mathcal V}((\xi,\mu),t)\| \notag \\
& \leq \eps (1+\|A\|)\leq 2\eps.  \notag
\end{align}
Moreover, $\|z-v\|=\|{\mathcal V}((\xi,\mu),t)-v\|\leq\eps<\rho$ and, by the choice in~\eqref{sceltadelt},
\begin{equation*}
	\|\zeta-u\| \leq \|\xi-u\|+t\|A\|\|z\|\leq \delta'+\delta'\|z\|\leq\delta'(1+\eps+\|v\|)<\rho.
\end{equation*}
Then, for $(\xi,\mu)\in B_{\delta'}(u,f(u))\cap {\rm epi}(f)\cap (\fix\times\R)$
and $t\in ]0,\delta']$, by \eqref{assAbs} and \eqref{zetnormcontr}, 
\begin{align*}
	f(\xi+t\hat {\mathcal V}((\xi,\mu),t)) &=f(\xi+t A{\mathcal V}((\xi,\mu),t))=f(\zeta)  \\
	&\leq f(\zeta+t(z-Az))+Ct\|z-Az\| \\
	& \leq f(\xi+t {\mathcal V}((\xi,\mu),t))+2C\eps t  \\
	&\leq \mu+(r+2C\eps)t.
\end{align*}
It follows $r+2C\eps \geq (f|_\fix)^\circ_\eps(u;v)$. By the arbitrariness of $r\in\R$, it follows
$$
f^\circ_\eps(u;v)\geq (f|_\fix)^\circ_{\eps}(u;v)-2C\eps.
$$
Finally, by the arbitrariness of $\eps$, letting $\eps\to 0^+$, we conclude
$$
f^\circ(u;v)\geq (f|_\fix)^\circ(u;v),
$$
which yields the assertion.
\end{proof}

\subsubsection{Proof of Theorem~\ref{abstr11}}
Let $u\in \fix$ be a critical point of the function $f|_{\fix}$, that is
$0\in \partial (f|_{\fix})(u)$. Then $(f|_\fix)^\circ(u;v)\geq 0$
for every $v\in \fix$, by Definition~\ref{defsub}. In light of Proposition~\ref{legamefix}, we have
$f^\circ(u;v)\geq 0$ for all $v\in \fix$. Finally, by 
Proposition~\ref{primostep}, we get $f^\circ(u;v)\geq 0$ for all $v\in X$,
namely $0\in \partial f(u)$, so that $u$ is a critical point of $f$.
\qed

\begin{remark}\rm
	In the proof of Proposition~\ref{legamefix}, given the deformation ${\mathcal V}$,
	in order to build a new deformation $\hat {\mathcal V}$ that satisfies the desired properties, the idea is to compose ${\mathcal V}$
with $A$ to ensure the values of $\hat {\mathcal V}$ are projected into $\fix$. This, at the end, requires the technical assumption~\eqref{assAbs}. 
In some concrete situations, the starting deformation ${\mathcal V}$, when restricted to $\fix\times\R$,
automatically brings the values into $\fix$, so that one can take directly $\hat {\mathcal V}={\mathcal V}|_{\fix\times\R}$. See,
for instance, the argument in the proof of Lemma~\ref{subdcaract}.  
\end{remark}

\section{A second abstract result}
\label{secondab}

Let $G$ be a compact Lie group acting linearly on a real Banach space $(X,\|\cdot\|)$.
We assume that for every $u\in\fix$ there exists a $G$-invariant dense vectorial subspace $V_u$
of $X$ such that the following conditions are satisfied
\begin{align}
		\label{crucineq0}
	\forall u\in \fix,\,\,\forall v\in V_u: & \quad \text{the directional derivative $f'(u)v$ exists},	\\		
	\label{crucineq1}
 \forall u\in \fix,\,\,\forall v\in V_u\cap\fix: & \quad (f|_{\fix})^\circ(u;v) \leq f'(u)v, \\
\label{crucineq2}
  \forall u\in \fix,\,\,\forall v\in V_u: &\quad f^\circ(u;v) \leq f'(u)v.
\end{align}
Furthermore, for every $u\in\fix$, we assume that
\begin{equation}
	\label{suspazio}
V_u=\bigcup_{j\in J}C_j,\quad\text{$C_j$ is convex, closed, $G$-invariant with $f'(u)|_{C_j}$ continuous}
\end{equation}
\vskip4pt
\noindent
The main result of the section is the following

\begin{theorem}
	\label{abstr22}
Let $G$ be a compact Lie group acting linearly on a Banach space $X$ and let
$f:X\to \bar\R$ be a $G$-invariant functional. Assume that conditions
\eqref{crucineq0}-\eqref{suspazio} are satisfied and let $u\in \fix$ with $f(u)\in\R$ 
be a critical point of $f|_{\fix}$. Then 
$$
\forall v\in V_u:\quad f'(u)v=0.
$$
Furthermore, $\partial f(u)=\{0\}$ provided that $\partial f(u)$ is nonempty.
\end{theorem}

In the statement of Theorem~\ref{abstr22} we are not explicitly 
assuming any global regularity on the functional $f$. In the last
section of the paper we shall apply it to a class of (nonconvex) lower semi-continuous functionals.
This achievement could not be reached through previous non-smooth 
versions of Palais' symmetric criticality principle in the literature as they
require $f$ to be the sum of a locally Lipschitz functions with a lower 
semi-continuous convex function~\cite{krisvarga}.

\subsubsection{Proof of Theorem~\ref{abstr22}}
Let $u\in \fix$ be a critical point of $f|_{\fix}$, 
$0\in \partial (f|_{\fix})(u)$. Then $(f|_\fix)^\circ(u;v)\geq 0$,
for every $v\in \fix$. In particular, $(f|_\fix)^\circ(u;v)\geq 0$,
for every $v\in V_u\cap \fix$. In turn, in light of~\eqref{crucineq1}, by exploiting
the linearity of the map $\{v\mapsto f'(u)v\}$, we get
$$
\forall v\in V_u\cap \fix:\quad f'(u)v=0.
$$
Let now $\eps>0$, $j\geq 1$ and consider the set
$$
D^j_\eps=\{v\in C_j: f'(u)v\leq -\eps\}.
$$
Assume, by contradiction, that $D_\eps^j\not=\emptyset$. Of course, 
$D_\eps^j$ is convex, closed and $G$-invariant in $X$ (recall that assumption
\eqref{suspazio} holds). In turn, $A(D_\eps^j)\subset D_\eps^j\cap\fix$. 
We deduce that $D_\eps^j\cap\fix\not=\emptyset$, which yields a contradiction. By the 
arbitrariness of $\eps$ and $j$,
$$
\{v\in V_u: f'(u)v<0\}=\bigcup_{j\in J}\bigcup_{\eps>0} D_\eps^j=\emptyset,
$$
yielding 
\begin{equation}
	\label{firstconclu}
\forall v\in V_u:\quad f'(u)v=0.
\end{equation}
This proves the first assertion. Concerning the second assertion,
assume that $\partial f(u)$ is not empty and let $\alpha\in X^*$ with $\alpha\in\partial f(u)$. Then,
in light of~\eqref{crucineq2} and \eqref{firstconclu}, we have
$$
\forall v\in V_u:\quad \langle \alpha,v\rangle\leq f^\circ(u;v)\leq f'(u)v=0.
$$
Then, 
$$
\forall v\in V_u:\quad \langle \alpha,v\rangle=0.
$$
Taking into account that $V_u$ is dense in $X$, we conclude
$\alpha=0$. Hence $\partial f(u)\subseteq\{0\}$.
\qed
\smallskip

\section{A direct approach}
\label{concretesec}

The main goal of this section is that of showing how a direct (or, say, unrestricted) approach
in the search of symmetric critical points can also be outlined. In this approach, 
the compactness of the Palais-Smale sequences in a suitable space (relevant for applications to PDEs)
is achieved by exploiting the symmetry properties of the functional instead
of restricting the functional to a space of symmetric functions. After recalling an
abstract symmetrization framework due to \cite{jvsh1}, we shall state and prove a general abstract result 
by using the main result from \cite{sq-ccm}.

\subsection{Abstract symmetrization}
\label{abstractsymmetrrr}
We recall a definition from~\cite{jvsh1}.
\vskip1pt
\noindent
Let $X$ and $V$ be two Banach spaces and $S\subset X$.
We consider two maps $*:S\to S$, $u\mapsto u^*$ 
({\em symmetrization map}) and $h:S\times {\mathcal H}_*\to S$,
$(u,H)\mapsto u^H$ ({\em polarization map}), where ${\mathcal H}_*$ 
is a path-connected topological space. We assume that the following 
conditions hold:
\begin{enumerate}
 \item $X$ is continuously embedded in $V$;
 \item $h$ is a continuous mapping;
\item for each $u\in S$ and $H\in {\mathcal H}_*$ it holds $(u^*)^H=(u^H)^*=u^*$ and $u^{HH}=u^H$;
\item there exists a sequence $(H_m)$ in ${\mathcal H}_*$ such that, for $u\in S$, $u^{H_1\cdots H_m}$ converges
to $u^*$ in $V$;
\item for every $u,v\in S$ and $H\in {\mathcal H}_*$ it holds
$\|u^H-v^H\|_V\leq \|u-v\|_V$.
\end{enumerate}
Furthermore $*:S\to S$ can be extended to the whole space $X$ by 
setting $u^*:=(\Theta(u))^*$ for all $u\in X$, where $\Theta:(X,\|\cdot\|_V)\to (S,\|\cdot\|_V)$ is
a Lipschitz function such that $\Theta|_{S}={\rm Id}|_{S}$. 

\subsection{Compactness of Palais-Smale sequences}

The main result of the section is the following 

\begin{theorem}
	\label{main}
Let $X$ be a reflexive Banach space, $S\subset X$ and $*:S\to S$ a symmetrization 
satisfying the requirements of the abstract symmetrization framework.
Let $V$ and $W$ be Banach spaces containing $X$ such that 
\begin{enumerate}
\item[(a)] the injections $X\overset{i}{\hookrightarrow} V \overset{i'}\hookrightarrow W$ are continuous;
\end{enumerate}
Let $f:X\to\R\cup\{+\infty\}$ be a
lower semi-continuous function satisfying~\eqref{keycond}.
Let $\Dis$ and $\Sf$ denote the closed unit ball and the sphere in $\R^m$ ($m\geq 1$) 
respectively and $\Gamma_0\subset C(\Sf,X)$. Let us define 
$$
\Gamma=\big\{\gamma\in C(\Dis,X):\,\,\,\gamma|_{\Sf}\in \Gamma_0\big\}.
$$
Assume that
$$
+\infty>c=\inf_{\gamma\in\Gamma}\sup_{\tau\in \Dis} f(\gamma(\tau))
>\sup_{\gamma_0\in \Gamma_0}\sup_{\tau\in \Sf} f(\gamma_0(\tau))=a,
$$
and that the following conditions hold
\begin{enumerate}
\item[(b)] for all $H\in {\mathcal H}_*$, for all $u\in S$:\,\, $f(u^H)\leq f(u)$;
\item[(c)] for all $\gamma\in\Gamma$ there exists $\hat\gamma\in\Gamma$ and $H_0\in {\mathcal H}_*$ such that
$$
\hat\gamma(\Dis)\subset S,\,\,\quad 
f\circ\hat\gamma\leq f\circ\gamma,\,\,\quad
\text{$\hat\gamma|_{\Sf}^{H_0}\in\Gamma_0$}.
$$
\end{enumerate}
Assume furthermore that 
\begin{enumerate}
\item[(d)] for each Palais-Smale sequence $(u_h)\subset X$ for $f$, $(u_h)$ is bounded in $X$;
\item[(e)] for each Palais-Smale sequence $(u_h)\subset X$ for $f$, $(u_h^*)$ converges in $W$. 
\end{enumerate}
Then there exist $u\in X$ and a Palais-Smale sequence $(u_h)\subset X$ for $f$ at level $c$ such that
\begin{enumerate}
\item $u_h\to u$ weakly in $X$, as $h\to\infty$; 
\item $u_h\to u$ strongly in $W$, as $h\to\infty$;
\item if {\rm (4)-(5)} of the abstract framework
hold with $W$ in place of $V$, then $u=u^*$ in $W$. 
\end{enumerate}
\end{theorem}

\vskip4pt
\noindent
The theorem states the existence, under suitable assumptions, of a Palais-Smale sequence
$(u_h)$ in $X$ which is convergent in a Banach $W$ larger than $X$ to a symmetric element. In particular, of course, it does not claim the compactness
of $(u_h)$ in the original space $X$.
Notice also that, in the cases where the space $X$ is a Sobolev space defined over a
smooth bounded domain $\Omega\subset\R^N$ and $V,W$ are subcritical $L^p(\Omega)$ spaces then, by 
Rellich Theorem, the injection $X\hookrightarrow W$
is compact and the stated strong convergence in $W$ automatically follows from the weak convergence in $X$.
On the other hand, this is not the case for PDEs
which present loss of compactness, such as those set on an unbounded domain. In this sense, Theorem~\ref{main}
allows to avoid restricting the concrete functional to the space of radial functions and then use Palais'
symmetric criticality principle studied in the previous sections. It also provides an alternative to concentration-compactness
arguments \cite{lions1,lions2} under symmetry assumptions.
It is quite easy to realize that, whenever the conclusion of the theorem holds, namely
$u_h\to u$ weakly in $X$ and strongly in $W$, then it is often the case that, in turn,
$u_h\to u$ strongly in $X$ (using $f(u_h)\to c$ and $|df|(u_h)\to 0$ as $h\to\infty$). 
See, for instance, the compactness argument inside the proof of Theorem~\ref{appl}.
For an application of Theorem~\ref{main} in the case $p=2$ and $f\in C^1$ (the 
semi-linear case), see~\cite[Theorem 4.5]{jvsh1}.
In a similar fashion, applications to lower semicontinuous functionals can be given,
by arguing as in \cite{sq-ccm}, up to suitable necessary modifications.
As a meaningful concrete framework where Theorem~\ref{main} 
applies one can think, for instance, to the case where, for $p<m<p^*$,
$$
W^{1,p}(\R^N)=S=X\overset{i}{\hookrightarrow} V=L^p\cap L^{p^*}(\R^N)=V
\overset{i'}{\hookrightarrow} W=L^m(\R^N),
$$
with $i,i'$ continuous injections. The polarization and symmetrization functions are defined as $u^H=|u|^H$ and
$u^*=|u|^*$ and (1)-(5) are satisfied (cf~\cite{jvsh1}).

\subsection{Proof of Theorem~\ref{main}}

In \cite{sq-ccm}, the author proved the following

\begin{theorem}
\label{mpconc-symm}
Let $X$ and $V$ be two Banach spaces, $S\subset X$, $*$ and ${\mathcal H}_*$ 
satisfying the requirements of the abstract symmetrization framework. 
Let $f:X\to\R\cup\{+\infty\}$ a
lower semi-continuous function satisfying $\refe{keycond}$.
Let $\Dis$ and $\Sf$ denote the closed unit ball and the sphere in $\R^m$ ($m\geq 1$) 
respectively and $\Gamma_0\subset C(\Sf,X)$. Let us define 
$$
\Gamma=\big\{\gamma\in C(\Dis,X):\,\,\,\gamma|_{\Sf}\in \Gamma_0\big\}.
$$
Assume that
$$
+\infty>c=\inf_{\gamma\in\Gamma}\sup_{\tau\in \Dis} f(\gamma(\tau))
>\sup_{\gamma_0\in \Gamma_0}\sup_{\tau\in \Sf} f(\gamma_0(\tau))=a,
$$
and that
$$
\forall H\in {\mathcal H}_*,\,\, \forall u\in S:\quad
f(u^H)\leq f(u).
$$
Then, for every $\eps\in(0,(c-a)/3)$, every $\delta>0$ and $\gamma\in \Gamma$ such that
$$
\sup_{\tau\in \Dis}f(\gamma(\tau))\leq c+\eps,\quad \gamma(\Dis)\subset S,\quad
\text{$\gamma|_{\Sf}^{H_0}\in\Gamma_0$ for some $H_0\in {\mathcal H}_*$},
$$
there exists $u\in X$ such that
\begin{equation}
\label{conclusmmm}
 c-2\eps\leq f(u)\leq c+2\eps,\quad
|df|(u)\leq 3\eps/\delta,\quad
\|u-u^*\|_V\leq 3((1+C_\Theta)K+1)\delta,
\end{equation}
being $K$ the norm of the embedding map $i:X\to V$ 
and $C_\Theta$ the Lipschitz constant of $\Theta$.
\end{theorem}

\subsection{Proof of Theorem~\ref{main} concluded}
By the definition of the minimax value $c$,
by choosing $\eps=\eps_h=h^{-2}$, for each $h\geq 1$
there exists a curve $\gamma_h\in\Gamma$ such that 
$$
\sup_{\tau\in \Dis}f(\gamma_h(\tau))\leq c+\frac{1}{h^2}.
$$
Then, by virtue of assumption (c), we can find a curve $\hat\gamma_h\in\Gamma$ and $H_0^h\in {\mathcal H}_*$ such that
$$
\sup_{\tau\in \Dis}f(\hat\gamma_h(\tau))\leq c+\frac{1}{h^2},\quad 
\hat\gamma_h(\Dis)\subset S,\quad \hat\gamma_h|_{\Sf}^{H_0^h}\in\Gamma_0.
$$
Choose also $\delta=\delta_h=h^{-1}$. In turn, recalling assumption (b), by Theorem~\ref{mpconc-symm} there exists a sequence 
$(u_h)\subset X$ such that $f(u_h)\to c$ and $|df|(u_h)\to 0$ as $h\to\infty$, with the additional information that
\begin{equation}
\label{conclusmmm-conc}
\|u_h-u^*_h\|_V\to 0,\quad\text{as $h\to\infty$}.
\end{equation}
In particular, by Definition~\ref{defps}, $(u_h)$ is a Palais-Smale sequence for $f$ at the level $c$. By assumption (d), $(u_h)$ is bounded in $X$.
Since $X$ is reflexive, there exists $u\in X$ such that, up to a subsequence, $u_h\to u$ weakly in $X$, as $h\to\infty$.
Consider now the sequence $(u_h^*)$ of abstract symmetrizations of $u_h$. By assumptions (a) and (e), it follows that $(u_h^*)$ converges strongly
in $W$ and $V$ is continuously embedded into $W$. In particular, there exists $v\in X$ such that
$u_h^*\to v$ in $W$ as $h\to\infty$ and, for all $h\geq 1$,
$$
\|u_h-v\|_{W}\leq \|u_h-u^*_h\|_{W}+\|u^*_h-v\|_{W}\leq C\|u_h-u^*_h\|_{V}+\|u^*_h-v\|_{W},
$$
which yields $u_h\to v$ strongly in $W$, for $h\to\infty$, on account of~\eqref{conclusmmm-conc}. Then, of course, we deduce
that $v=u$, which yields $u_h\to u$ strongly in $W$. Finally, assume that conditions (4)-(5) of the abstract symmetrization framework
holds also with $W$ in place of $V$. Then it is readily seen that $\|z^*-w^*\|_W\leq C_\Theta\|z-w\|_W$ for all $z,w\in X$
(cf.~\cite[Remark 2.1]{sq-ccm}). In turn, for all $h\geq 1$, we get
\begin{align*}
\|u^*-u\|_W &\leq \|u^*-u_h^*\|_W+\|u^*_h-u_h\|_W+\|u_h-u\|_W \\
&\leq (1+C_\Theta)\|u_h-u\|_W+C\|u^*_h-u_h\|_V,
\end{align*}
which yields $u=u^*$ in $W$, by virtue of~\eqref{conclusmmm-conc} and the limit $u_h\to u$ in $W$.
\qed

\section{Application to quasi-linear PDEs}
\label{concrete}

In this section, we shall consider some applications of the abstract Palais 
criticality principle developed in Section~\ref{secondab}.

\subsubsection{Compatible invariant domains}

In the following, the symbol $G_N$ will denote a subgroup of ${\mathscr O}(N)$, 
the orthogonal group over $\R^N$, with $N\geq 2$. According to \cite[Definition 1.22]{willem},
we consider the following

\begin{definition}
	For every $y\in\R^N$ and $r>0$, we set
	$$
	{\mathscr M}(y,r,G_N):=\sup\big\{n\in\N:\exists g_1,\dots,g_n\in G_N:i\neq j\Rightarrow B_r(g_iy)\cap B_r(g_jy)=\emptyset\big\}.
	$$
	An open (possibly unbounded) subset $\Omega$ of $\R^N$ is said to be invariant provided that $g\Omega=\Omega$,
	for all $g\in G_N$.	An invariant subset $\Omega$ is said to be compatible
	with $G_N$ provided that
	$$
	\lim_{\overset{|y|\to\infty}{{\rm dist}(y,\Omega)\leq r}} {\mathscr M}(y,r,G_N)=\infty
	$$
	for some positive number $r$.
\end{definition}

\noindent
Throughout the rest of the paper we shall assume that
$\Omega\subset\R^N$ with $N\geq 2$ is a (possibly unbounded) smooth domain 
which is invariant under the action of $G_N$. Let 
$$
X=W^{1,p}_0(\Omega), \,\,\, 1<p<N,\,\,\,
\quad \|u\|_{1,p}=\big\{\|u\|_p^p+\|Du\|_p^p\big\}^{1/p}. 
$$
The action of $G_N$ on $W^{1,p}_0(\Omega)$ is defined in a standard fashion by setting
$$
\forall u\in W^{1,p}_0(\Omega),\,\,\forall g\in G_N:
\quad (gu)(x):=u(g^{-1}x)\quad\text{for a.e. $x\in\Omega$}.
$$
We shall denote by $\fixs$ the set of fixed points $u$ of $X$ with respect to the action of $G_N$, namely
$gu=u$ for all $g\in G_N$.
\vskip3pt
\noindent
By virtue of \cite[Proposition 4.2]{koba} we have the following
\begin{lemma}
	\label{compsobemb}
	Assume that $\Omega$ is compatible with $G_N$. Then the embeddings
	\begin{equation}
	\label{immersionicompattt}
	\fixs\hookrightarrow L^m(\Omega),\,\,\quad \text{for all $p<m<p^*$}
\end{equation}
	are compact.
\end{lemma}

\begin{remark}\rm
	\label{partimpcase}
A particular, but important, case is contained in Lemma~\ref{compsobemb}, namely
$$
\Omega=\R^N,\quad
G_N:={\mathscr O}(N_1)\times {\mathscr O}(N_2)\times\cdots {\mathscr O}(N_\ell),
\quad N=\sum_{j=1}^\ell N_j, \quad \ell\geq 1,\,\, N_j\geq 2.
$$
In fact, $\R^N$ is compatible with this $G_N$ (cf.~\cite[Corollary 1.25]{willem}). Then Lemma~\ref{compsobemb} yields
$$
{\rm Fix}_{W^{1,p}(\R^N)}({\mathscr O}(N_1)\times {\mathscr O}(N_2)\times\cdots {\mathscr O}(N_\ell))\hookrightarrow L^m(\Omega),\quad\,\,
\text{for all $p<m<p^*$},
$$
with compact injection. In particular, ${\rm Fix}_{W^{1,p}(\R^N)}({\mathscr O}(N))\hookrightarrow L^m(\Omega)$ with compact injection.
\end{remark}

\subsubsection{Invariant quasi-linear functionals}
\label{invffffunct}
Let $f:W^{1,p}_0(\Omega)\to\R\cup\{+\infty\}$ be the functional 
\begin{equation}
	\label{funzionale}
f(u)=J(u)+I(u),\,\,\qquad 
J(u):=\int_\Omega j(u,|Du|),\,\,\,\, I(u):=\int_\Omega \frac{|u|^p}{p}-\int_\Omega \frac{|u|^q}{q},
\end{equation}
where $p<q<p^*$. Notice that $I:W^{1,p}_0(\Omega)\to\R$ is a $C^1$ functional while, on the contrary,
$J:W^{1,p}_0(\Omega)\to\R\cup\{+\infty\}$ is, in general, merely lower semi-continuous, by Fatou's lemma, if $j\geq 0$. Moreover, 
$f$ is invariant under $G_N$, being invariant under the action of ${\mathscr O}(N)$, that is
$$
\forall u\in W^{1,p}_0(\Omega),\,\,\forall g\in {\mathscr O}(N):\quad
f(gu)=f(u).
$$
We consider the following assumptions on $j$. We assume that, for every $s$ in $\R$,
\begin{equation}\label{j1}
\text{$\big\{t\mapsto j(s,t)\big\}$
is strictly convex and increasing}.
\end{equation}
Moreover, there exist a constant $\alpha_0>0$ and
a positive increasing function $\alpha\in C(\R)$
such that, for every $(s,t)\in\R\times\R^+$, it holds
\begin{equation}\label{j2}
\alpha_0 t^p\leq j(s,t)\leq \alpha(|s|)t^p.
\end{equation}
The functions $j_s(s,t)$ and $j_t (s,t)$
denote the derivatives of $j(s,t)$ with
respect to the variables $s$ and $t$
respectively.
Regarding the function $j_s(s,t)$,
we assume that there exist two positive increasing functions
$\beta,\gamma\in C(\R)$ and a positive constant $R$ such that 
\begin{equation}\label{j3}
|j_s(s,t)|\leq \beta(|s|)t^p,\qquad \text{
for every $s$ in $\R$ and all $t\in\R^+$},
\end{equation}
\begin{equation}\label{j32}
|j_t(s,t)|\leq \gamma(|s|)t^{p-1},\qquad \text{
for every $s$ in $\R$ and all $t\in\R^+$},
\end{equation}
\begin{equation}\label{j4}
\qquad \qquad  j_s(s,t)s\geq 0,
\qquad \text{for every $s$ in $\R$ with $|s|\geq R$ and all $t\in\R^+$.}
\end{equation}
It is readily seen that, without loss of generality, we can assume that $\gamma=\alpha$, up to a constant.
Furthermore, we assume that
there exist $R'>0$ and $\delta>0$ such that
\begin{equation}
\label{j5}
q j(s,t)-j_s(s,t)s-(1+\delta)j_t(s,t)t\geq 0,
\quad\text{for every $s\in\R$ with $|s|\geq R'$}
\end{equation}
and all $t\in\R^+$. Finally, it holds 
\begin{equation}
\label{j6}
\lim\limits_{|s|\to\infty} \frac{\alpha(|s|)}{|s|^{q-p}}=0.
\end{equation}
In the above assumptions, $f$ is merely lower semi-continuous. If
$\alpha$ is bounded, then $f$ becomes a continuous functional. Condition~\eqref{j4}
is typical for these problems and plays a significant r\v ole in the verification of the Palais-Smale 
condition and in the regularity theory (cf.~e.g.~\cite{sq-monog}).
Condition~\eqref{j5} allows the Palais-Smale sequences to be bounded in $W^{1,p}_0(\Omega)$, while
\eqref{j6} guarantees that $f$ admits a Mountain Pass geometry.

\subsubsection{Statement of the results}

The main result of the section is the following

\begin{theorem}
	\label{appl}
	Assume that $\Omega$ is compatible with $G_N$. 
	Then there exists a nontrivial solution $u\in \fixs$ to the quasi-linear problem	
	\begin{equation}
		\label{problemvarform}
	\int_\Omega j_t(u,|Du|)\frac{Du}{|Du|}\cdot D v +
		\int_\Omega j_s(u,|Du|)v+\int_\Omega |u|^{p-2}uv=\int_\Omega |u|^{q-2}uv,
	\end{equation}
	for every $v\in C^\infty_c(\Omega)$.
\end{theorem}

\vskip3pt
\noindent
Furthermore, let $\Omega=\R^N$ with $N\geq 2$. Then, we have the following

\begin{corollary}
	\label{appl-cor1}
	 Problem~\eqref{problemvarform} admits a nontrivial radial positive solution.
\end{corollary}

\begin{remark}\rm
	Assume $p=2$ and $j(s,t)=\frac{t^2}{2}$. Then, Corollary~\ref{appl-cor1} reduces to the 
	results due to {\sc Strauss}~\cite{strausscomp} (see, for instance, \cite[Theorem 1.29]{willem}).
\end{remark}

\begin{remark}\rm
In place of $|u|^{q-2}u$, more general nonlinearities $f(|x|,u)$ could be handled by Theorem~\ref{appl}.
For instance, if $\Omega=\R^N$ and $p<q<p^*$, one could assume that, for 
all $\eps>0$, there exists $C_\eps>0$ such that
\begin{equation}
	\label{crescita}
|f(|x|,s)|\leq \eps |s|^{p-1}+C_\eps |s|^{q-1},\quad\text{for a.e. $x\in\R^N$ and all $s\in\R$}.
\end{equation}
Let $W^{1,p}_{G}(\R^N)$ denote $\fix$ when $X=W^{1,p}(\R^N)$.
We claim that the map $\{u\mapsto f(|x|,u)\}$ is completely continuous from $W^{1,p}_{G}(\R^N)$ to its dual, as soon as the injection
of $W^{1,p}_{G}(\R^N)$ into $L^q(\R^N)$ is compact.
In fact, let $(u_h)\subset W^{1,p}_{G}(\R^N)$ be a bounded sequence. Then, up to a subsequence, 
$(u_h)$ converges weakly in $W^{1,p}(\R^N)$ and strongly in $L^q(\R^N)$ 
to some function $u\in W^{1,p}_{G}(\R^N)$. Let $\eps>0$ and let $C_\eps$ be the constant appearing
in~\eqref{crescita}. Then, it is readily checked that, for all $\eta>0$, there exist 
$R=R(\eta)>0$ and $C>0$ (independent of $\eps$ and $\eta$) with
\begin{equation*}
\sup_{h\geq 1}\sup_{\overset{v\in W^{1,p}_{G}(\R^N)}{\|v\|_{1,p}\leq 1}}
\Big|\int_{\R^N\setminus B(0,R)} (f(|x|,u_h)-f(|x|,u))v\Big|< C_\eps\eta+\eps C.
\end{equation*}
Therefore, taking into account that $\|f(|x|,u_h)-f(|x|,u)\|_{L^{q'}(B(0,R))}\to 0$ as $h\to\infty$,
the desired claim follows by letting first $h\to\infty$, then $\eta\to 0^+$ and, finally, $\eps\to 0^+$.
\end{remark}

\subsubsection{Some preparatory facts}

For all $u\in W^{1,p}_0(\Omega)$, define the space
\begin{equation}\label{defvu}
V_u=\big\{v\in W^{1,p}_0(\Omega)\cap L^\infty(\Omega):\,
u\in L^{\infty}(\{x\in\Omega:\,v(x)\not=0\})\big\}.
\end{equation}

The vector space $V_u$ is dense in $W^{1,p}_0(\Omega)$ (cf.~\cite{degzan}).
The following proposition,  easy to prove, shows that $V_u$ is a good test space
to differentiate non-smooth functionals under
suitable growth conditions. In particular, as a consequence, 
the abstract condition~\eqref{crucineq0} is fulfilled.

\begin{proposition}
	\label{derivate}
Assume conditions~\eqref{j2}, \eqref{j3} and \eqref{j4}. 
Then, for every $u\in W^{1,p}_0(\Omega)$ with $J(u)<+\infty$
and every $v\in V_u$ we have
$$
j_s(u,|Du|)v\in L^1(\Omega),\qquad
j_t(u,|Du|)\frac{Du}{|Du|}\cdot D v\in L^1(\Omega),
$$
with the agreement that $j_t(u,|Du|)\frac{Du}{|Du|}=0$ when $|Du|=0$ (in view of~\eqref{j4}). Moreover,
the function $\{t\mapsto J(u+tv)\}$ is of class $C^1$ and 
$$
J'(u)(v)=\int_{\Omega}j_t(u,|Du|)\frac{Du}{|Du|}\cdot Dv
+\int_{\Omega}j_s(u,|Du|)v.
$$
In particular,
$$
f'(u)(v)=\int_{\Omega}j_t(u,|Du|)\frac{Du}{|Du|}\cdot Dv
+\int_{\Omega}j_s(u,|Du|)v+\int_{\Omega}|u|^{p-2}uv-\int_{\Omega}|u|^{q-2}uv,
$$
for every $v\in V_u$.
\end{proposition}

\vskip2pt
\noindent
For all $u\in \fixs$ and any $j\geq 1$, let us now set
$$
C_j=\{v\in W^{1,p}_0(\Omega): \text{$|v|\leq j$ a.e. in $\Omega$ and $|u|\leq j$ a.e. where $v\neq 0$}\}.
$$
Then, we have the following

\begin{lemma}
	\label{splitting}
For all $u\in \fixs$ and $j\geq 1$, 
the set $C_j$ is convex, closed, $G_N$-invariant with $f'(u)|_{C_j}$ continuous and
\begin{equation}
	\label{decompositform}
V_u=\bigcup_{j=1}^\infty C_j.
\end{equation}
In particular $V_u$ is $G_N$-invariant.
\end{lemma}
\begin{proof}
	Of course each $C_j$ is convex. Moreover $C_j$ is $G$-invariant. In fact, if $v\in C_j$,
	then $|(gv)(x)|=|v(g^{-1}x)|\leq j$ a.e.~in $\Omega$, being the domain invariant under $G_N$. 
	Moreover, if $x\in\{gv\neq 0\}$, then $v(g^{-1}x)\neq 0$, so that
	$|u(x)|=|u(g^{-1}x)|\leq j$, since $u\in\fixs$. Hence $gv\in C_j$. Let us now prove that $C_j$ is closed. If $(v_h)\subset C_j$
	with $v_h\to v$ in $W^{1,p}_0(\Omega)$, then up to a subsequence $|v_h(x)|\leq j$ and $v_h(x)\to v(x)$, as
	$h\to\infty$, yielding $|v|\leq j$ a.e.~in $\Omega$. Moreover, if $A=\{v\neq 0\}$, then by the pointwise convergence,
	$$
	A\subset B=\bigcup_{h=1}^\infty\{v_h\neq 0\}.
	$$
	Since $\sup_B|u|\leq j$, it holds  $\sup_A |u|\leq \sup_B |u|\leq j$. Thus, each $C_j$ is closed.
	Let us now prove~\eqref{decompositform}. Let $v\in V_u$. Since $v\in L^\infty(\Omega)$, there exists $j_0\geq 1$ such that
	$|v|\leq j_0$, a.e. in $\Omega$. Since $u$ is uniformly bounded over $\{v\neq 0\}$,
	up to enlarging $j_0$, we may as well assume that $|u|\leq j_0$ over $\{v\neq 0\}$. Hence $u\in C_{j_0}$,
	proving the assertion, the converse inclusion being trivial.
	Finally, let us prove that $f'(u)|_{C_j}$ is continuous. To this aim, let $(v_h)\subset C_j$
	with $v_h\to v$ in $W^{1,p}_0(\Omega)$, as $h\to\infty$. For all $h\geq 1$, since $v_h\in V_u$, by 
	Proposition~\ref{derivate}, we get
	$$
	f'(u)v_h=\int_{\Omega}j_t(u,|Du|)\frac{Du}{|Du|}\cdot Dv_h
	+\int_{\Omega}j_s(u,|Du|)v_h+\int_{\Omega} |u|^{p-2}uv_h-\int_{\Omega} |u|^{q-2}uv_h.
	$$
Of course, the last two integrals converge to $\int_{\Omega} |u|^{p-2}uv$
and $\int_{\Omega} |u|^{q-2}uv$, respectively. Also
	\begin{align*}
	 j_t(u,|Du|)\frac{Du}{|Du|}\cdot Dv_h &\to j_t(u,|Du|)\frac{Du}{|Du|}\cdot Dv,\quad \text{a.e. in $\Omega$}, \\
	 j_s(u,|Du|)v_h &\to j_s(u,|Du|)v,\quad \text{a.e. in $\Omega$}. 
	\end{align*}
In addition, in light of the growth conditions on $j_t,j_s$, for some positive constant $M_j$,
	\begin{align*}
&  |j_t(u,|Du|)\frac{Du}{|Du|}\cdot Dv_h|\leq M_j |Du|^{p-1}|Dv_h|,\quad \text{a.e. in $\Omega$}, \\
& |j_s(u,|Du|)v_h|\leq M_j |Du|^p|v_h|\leq j M_j |Du|^p,\quad \text{a.e. in $\Omega$}.
\end{align*}
Then, Lebesgue dominated convergence theorem yields
$f'(u)v_h\to f'(u)v$, as $h\to\infty$. 
\end{proof}

Then we have the following

\begin{lemma}
	\label{subdcaract}
	Assume conditions~\eqref{j2}-\eqref{j32} and let 
	$u\in \fixs$ such that $J(u)<+\infty$. Then, the following facts hold:
	\begin{itemize}
		\item[{\rm (i)}] for every $v\in V_u$, we have
		$$
		J^\circ(u;v)\leq J'(u)v.
		$$
		\item[{\rm (ii)}] for every $v\in V_u\cap \fixs$, we have
		$$
		(J|_\fixs)^\circ(u;v)\leq J'(u)v.
		$$
	\end{itemize}
\end{lemma}
\begin{proof}
	The proofs of assertions (i) and (ii) are similar. 
	Hence, let us focus on the proof of (ii). Let $\eta>0$ with $J(u)<\eta$. Moreover, 
	let $v\in V_u\cap\fixs$ and $\eps>0$. Take now $r\in\R$ with
\begin{equation}
	\label{step1}
J'(u)v=\int_\Omega j_t(u,|Du|)\frac{Du}{|Du|}\cdot D v +
\int_\Omega j_s(u,|Du|)v<r.
\end{equation}
Let $H\in C^\infty(\R)$ be a cut-off function such that
$H(s)=1$ on $[-1,1]$, $H(s)=0$ outside $[-2,2]$ and
$|H'(s)|\leq 2$ on $\R$. Notice that, as $v\in V_u$, we also have $H(\frac{u}{k})v\in V_u$
and $H(\frac{z}{k})v\in V_z$ for all $z\in W^{1,p}_0(\Omega)$ and $k\geq 1$.
Then, by exploiting the growth conditions on $j_t$ and $j_s$, it is possible to prove that there exist 
$k\geq 1$ and $\delta>0$ (depending upon $k$) such that 
\begin{equation}
	\label{prima3-mod}
\Big\|H(\frac{z}{k})v-v\Big\|_{1,p}<\eps,
\end{equation}
as well as
\begin{align}
	\label{step3-mod}
J'(z+\theta H(\frac{z}{k})v) \Big(H(\frac{z}{k})v\Big)<r,
\end{align}
for all $z\in B(u,\delta)\cap J^\eta$ and $\theta\in [0,\delta)$. Since the map
$\{t\mapsto J(z+t H(\frac{z}{k})v)\}$ is of class $C^1$, by applying Lagrange theorem on $[0,t]$ and
taking into account~\eqref{step3-mod}, there exists $\theta\in [0,t]$ with
\begin{equation}
	\label{ccruxineqq}
 J(z+t H(\frac{z}{k})v)-J(z)=t J'(z+\theta H(\frac{z}{k})v) \Big(H(\frac{z}{k})v\Big)\leq rt,
\end{equation}
for every $z\in B(u,\delta)\cap J^\eta$ and all $t\in [0,\delta)$.
Up to reducing the value of $\delta>0$, we may also assume that $J(u)+\delta<\eta$. Notice that
$H(\frac{z}{k})v\in\fixs$ for all $v,z\in\fixs$, since 
$$
\forall g\in G_N:\,\,\,  g\big(H(\frac{z}{k})v\big)=H(\frac{gz}{k})gv=H(\frac{z}{k})v.
$$
Then, on account of inequality~\eqref{prima3-mod}, we are allowed to define the continuous function 
$$
{\mathcal H}: B_\delta (u,J(u))\cap {\rm epi}(J)\cap (\fixs\times\R)\times ]0,\delta]\to B_\eps(v)\cap \fixs,
$$
by setting
$$
{\mathcal H}((z,\mu),t):=H(\frac{z}{k})v.
$$
Notice that, for all $(z,\mu)\in B_\delta (u,J(u))\cap {\rm epi}(J)$, we have
$z\in B(u,\delta)\cap J^\eta$. Hence, by inequality~\eqref{ccruxineqq}, we have
$$
J(z+t{\mathcal H}((z,\mu),t))\leq J(z) +rt\leq \mu+rt.
$$
whenever $(z,\mu)\in B_\delta(u,J(u))\cap {\rm epi}(J)$ and $t\in ]0,\delta]$. Then, according to
Definition~\ref{effe0}, we can conclude that $({J|_\fixs})_\eps^\circ(u;v)\leq r$. 
By the arbitrariness of $r$, it follows that
$$
({J|_\fixs})_\eps^\circ(u;v)\leq \int_\Omega j_t(u,|Du|)\frac{Du}{|Du|}\cdot Dv +
\int_\Omega j_s(u,|Du|)v.
$$
By the arbitrariness of $\eps$, we get $({J|_\fixs})^\circ(u;v)\leq J'(u)v$,
for all $v\in V_u$. 
\end{proof}

\begin{remark}\rm
Under the assumptions of Lemma~\ref{subdcaract}, assuming that $\partial J(u)\neq\emptyset$, 
then $\partial J(u)=\{\alpha\}$ for some $\alpha\in W^{-1,p'}(\Omega)$ such that
	\begin{equation*}
		\forall v\in V_u:\quad \int_\Omega j_t(u,|Du|)\frac{Du}{|Du|}\cdot D v +
		\int_\Omega j_s(u,|Du|)v=\langle \alpha,v\rangle.
	\end{equation*}
	On account of the definition of $\partial f(u)$, it is sufficient to combine 
	(i) of Lemma~\ref{subdcaract}, the linearity of $\{v\mapsto J'(u)v\}$ and, of course, the density of 
	$V_u$ in $W^{1,p}_0(\Omega)$. In a similar fashion, assuming $\partial J|_{\fixs}(u)\neq\emptyset$, 
	then $\partial J|_{\fixs}(u)=\{\beta\}$ for some $\beta\in (\fixs)'$ such that
		\begin{equation*}
			\forall v\in V_u\cap \fixs:\quad \int_\Omega j_t(u,|Du|)\frac{Du}{|Du|}\cdot D v +
			\int_\Omega j_s(u,|Du|)v=\langle \beta,v\rangle.
		\end{equation*}
		This follows by (ii) of Lemma~\ref{subdcaract}, the linearity of $\{v\mapsto J'(u)v\}$ and, finally, by the 
		density of $V_u\cap\fixs$ in $\fixs$, which is stated in Proposition~\ref{densitypro}.
\end{remark}

Finally, returning to the functional $f$ defined in~\eqref{funzionale}, we have the following

\begin{corollary}
	\label{subdcaract-cor}
	Assume conditions~\eqref{j2}-\eqref{j32} and let 
	$u\in \fixs$ such that $f(u)<+\infty$. Then, the following facts hold:
	\begin{itemize}
		\item[{\rm (i)}] for every $v\in V_u$, we have
		$$
		f^\circ(u;v)\leq f'(u)v.
		$$
		\item[{\rm (ii)}] for every $v\in V_u\cap \fixs$, we have
		$$
		(f|_\fixs)^\circ(u;v)\leq f'(u)v.
		$$
	\end{itemize}
\end{corollary}
\begin{proof}
By~\eqref{funzionale}, it is $f=J+I$ with $I:W^{1,p}_0(\Omega)\to\R$ of class $C^1$. 
Then $I^\circ(u;v)=I'(u)v$ for any $v\in V_u$ and $(I|_\fixs)^\circ(u;v)=I'(u)v$ for any $v\in V_u\cap \fixs$.
Of course, we have $J(u)<+\infty$. Then, by combining~\cite[Theorem 5.1]{campad} with Lemma~\ref{subdcaract}, we have
\begin{align*}
\forall v\in V_u:\quad   f^\circ(u;v)&=(J+I)^\circ(u;v) \\
& \leq J^\circ(u;v)+I^\circ(u;v) \\
&\leq J'(u)v+I'(u)v=f'(u)v.
\end{align*}
In a similar fashion, we have
\begin{align*}
\forall v\in v\in V_u\cap \fixs:\quad    (f|_\fixs)^\circ(u;v)&=(J|_\fixs+I|_\fixs)^\circ(u;v)  \\
&\leq (J|_\fixs)^\circ(u;v)+(I|_\fixs)^\circ(u;v)  \\
&\leq J'(u)v+I'(u)v=f'(u)v,
\end{align*}
concluding the proof.
\end{proof}

In light of the previous facts, we have the following

\begin{corollary}
	\label{corconcrabs2}
Let $f:W^{1,p}_0(\Omega)\to\R\cup\{+\infty\}$ be the $G_N$-invariant functional defined in~\eqref{funzionale}
which satisfy assumptions indicated in Section~\ref{invffffunct}. Let $u\in \fixs$ with $f(u)\in\R$ 
be a critical point of $f|_{\fixs}$, that is $0\in\partial f|_{\fixs}(u)$. Then 
$$
\forall v\in V_u:\quad f'(u)v=0.
$$
Furthermore, $\partial f(u)=\{0\}$ provided that $\partial f(u)$ is nonempty.
\end{corollary}
\begin{proof}
It is sufficient to apply Theorem~\ref{abstr22}, since conditions \eqref{crucineq0}-\eqref{suspazio} are 
satisfied in view of Proposition~\ref{derivate}, Lemma~\ref{splitting} and Corollary \ref{subdcaract-cor}.
\end{proof}

Next we state some technical facts, necessary to the proof of the main result, Theorem~\ref{appl}.

\begin{proposition}
	\label{chiavesemicontfunct}
	It holds
	\begin{equation*}
	\forall (u,\xi)\in\epi{f|_\fixs}:\quad
	f(u)<\xi\,\,\,\Longrightarrow\,\,\,\ws{\gra{f|_\fixs}}(u,\xi)=1.
	\end{equation*}
\end{proposition}
\begin{proof}
	Notice first that~\cite[Theorem 3.11]{pelsqu} extends to the case where $H^1_0(\Omega)$ is
	substituted by $W^{1,p}_0(\Omega)$, $\Omega$ is possibly unbounded and the growths~\eqref{j2}-\eqref{j32} are assumed.
	Since $f=J+I$ with $I\in C^1(\fixs,\R)$, the assertion follows by~\cite[Proposition 3.7]{pelsqu} 
	with $X=\fixs$ by using~\cite[Theorem 3.11]{pelsqu} with $J$ replaced by the restriction $J|_\fixs$. In fact, it is sufficient 
	to notice that for the deformation (for $\delta,\eta>0$ and $k\geq 1$)
	$$
	{\mathscr H}(z,t)=(1-t)z+t T_k(z),\qquad z\in B(u,\delta)\cap J^\eta,\,\, t\in[0,\delta],
	$$ 
	making the job in~\cite[Theorem 3.11]{pelsqu} it is ${\mathscr H}(z,t)\in \fixs$ 
	for $z\in B(u,\delta)\cap\fixs$.
\end{proof}

\begin{proposition}
	\label{densitypro}
Let	$u\in \fixs$. Then the subspace $V_u\cap \fixs$ is dense in $\fixs$.
\end{proposition}
\begin{proof}
	Let $v\in \fixs$ and take $\eps>0$. Let $T_k:\R\to\R$ be the Lipschitz function
	such that $T_k(s)=s$, for $|s|\leq k$, and $T_k(s)=ks|s|^{-1}$, for $|s|\geq k$. 
	Then $T_k(v)\in W^{1,p}_0(\Omega)\cap L^\infty(\Omega)$ and $T_k(v)\in\fixs$.
	By Lebesgue's theorem, there exists $k\geq 1$ such that $\|v-T_k(v)\|_{1,p}<\eps/2$.
	For this $k\geq 1$, for all $h\geq 1$, consider now $v_h=H(u/h)T_k(v)$, where $H$ is defined as in 
	the proof of Lemma~\ref{subdcaract}.
	Then, $v_h\in W^{1,p}_0(\Omega)\cap L^\infty(\Omega)$, $v_h\in\fixs$ and, of course, 
	$u\in L^\infty(\{v_h\neq 0\})$. In turn, $v_h\in V_u\cap\fixs$, $v_h(x)\to T_k(v)(x)$,
	$|v_h(x)|\leq |T_k(v)|\leq |v(x)|$ for a.e.~$x\in \Omega$,
	$D_j v_h(x)\to D_jT_k(v)(x)$ a.e.~in $\Omega$ as $h\to\infty$ and there exists a positive constant $C$
	(depending upon $k$), such that
	\begin{align*}
	|D_j v_h(x)| &=|h^{-1}H'(u/h)D_ju(x) T_k(v)(x)+H(u/h)D_j v(x)\chi_{\{|v|\leq k\}}(x)| \\
	& \leq C|D_j u(x)|+C|D_j v(x)|,\quad\text{for a.e.~$x\in \Omega$}.
\end{align*}
	By Lebesgue's theorem, we find $h\geq 1$ so large that $\|T_k(v)-v_h\|_{1,p}<\eps/2$. 
	Then, we conclude that $\|v_h-v\|_{1,p}\leq \|v_h-T_k(v)\|_{1,p}+\|T_k(v)-v\|_{1,p}<\eps$
	with $v_h\in V_u\cap\fixs$.
\end{proof}

\begin{proposition}\label{linkkkfix}
Let	$u\in \fixs$. Then it holds
\begin{align*}
|df|_\fixs|(u)  &\geq \sup\Big\{\int_\Omega  j_t(u,|Du|)\frac{Du}{|Du|}\cdot D v+
\int_\Omega j_s(u,|Du|)v\\
& \,\,\, +\int_\Omega |u|^{p-2}uv-\int_\Omega |u|^{q-2}uv:\,\, v\in V_u\cap\fixs,\,\,\|v\|_{1,p}
\leq 1\Big\}.
\end{align*}
Moreover,
\begin{align}
\label{stimaslope}
\int_\Omega j_t(u,|Du|)|Du|+\int_\Omega j_s(u,|Du|)u \leq |dJ|_\fixs|(u)\|u\|_{1,p}.
\end{align}
In turn, if $|dJ|_\fixs|(u)<\infty$, then
$j_t(u,|Du|)|Du|\in L^1(\Omega)$ and $j_s(u,|Du|)u\in L^1(\Omega)$.
\end{proposition}
\begin{proof}
Concerning the first assertion, notice that~\cite[Propositions 4.5 and 6.2]{pelsqu} extend to the case where $H^1_0(\Omega)$ is
substituted by $W^{1,p}_0(\Omega)$, $\Omega$ is possibly unbounded and \eqref{j2}-\eqref{j32} are assumed.		
The assertion follows as in \cite[Propositions 4.5 and 6.2]{pelsqu} applied with $f$ and $J$ substituted by
$f|_{\fixs}$ and $J|_{\fixs}$. It is enough to notice that, fixed $v\in V_u\cap\fixs$, the deformation 
(for $\delta>0$ and $k\geq 1$)
$$
{\mathscr H}(z,t)=z+\frac{t}{1+\eps} H\big(\frac{z}{k}\big)v,\qquad z\in B(u,\delta),\,\, t\in[0,\delta],
$$
that makes the job in the proof of~\cite[Proposition 4.5]{pelsqu} satisfies ${\mathscr H}(z,t)\in \fixs$ 
for every $z\in B(u,\delta)\cap\fixs$. 
Concerning the second assertion, notice that~\cite[Lemma 4.6]{pelsqu} extends to the case where $H^1_0(\Omega)$ is
substituted by $W^{1,p}_0(\Omega)$, $\Omega$ is possibly unbounded and \eqref{j2}-\eqref{j32} hold.
Furthermore, the deformation (for $\delta>0$ and $k\geq 1$)
$$
{\mathscr H}(z,t)=z-\frac{t}{\|T_k(u)\|_{1,p}(1+\eps)} T_k(z),\qquad z\in B(u,\delta),\,\, t\in[0,\delta],
$$
that works inside the proof of~\cite[Lemma 4.6]{pelsqu} is again such that 
${\mathscr H}(z,t)\in \fixs$ whenever one takes $z\in B(u,\delta)\cap\fixs$.
\end{proof}

\begin{proposition}\label{testenlarging}
Let $w\in (\fixs)'$ and let $u\in \fixs$ be such that
\begin{equation*}
\int_{\Omega}j_s(u,|Du|)v+\int_{\Omega}
j_t(u,|Du|)\frac{Du}{|Du|}\cdot Dv  
+\int_\Omega |u|^{p-2}uv=\int_\Omega |u|^{q-2}uv+\langle w,v\rangle,
\end{equation*}
for every $v\in V_u\cap \fixs$.
Moreover, suppose that $j_t(u,|Du|)|Du|\in L^1(\Omega)$,
and that there exist  $v\in \fixs$ and $\eta\in L^1(\Omega)$ such that
\begin{equation}\label{control2}
j_s(u,|Du|)v\geq\eta,
\qquad
j_t(u,|Du|)\frac{Du}{|Du|}\cdot Dv\geq\eta.
\end{equation}
Then $j_s(u,|Du|)v\in L^1(\Omega)$, 
$j_t(u,|Du|)\frac{Du}{|Du|}\cdot Dv\in L^1(\Omega)$ 
and $v$ is an admissible test.
\end{proposition}
\begin{proof}
It is a simple adaptation of~\cite[Theorem 4.8]{pelsqu} to the invariant framework.
\end{proof}

We are now ready to provide the proof of Theorem~\ref{appl}.

\subsubsection{Proof of Theorem~\ref{appl}}
By Proposition~\ref{derivate} and Corollary~\ref{subdcaract-cor}, $f$ satisfies condition~\eqref{crucineq0} and conditions 
\eqref{crucineq1} and \eqref{crucineq2} of Theorem~\ref{abstr22}. Furthermore, by Lemma~\ref{splitting},
condition~\eqref{suspazio} is satisfied as well. We now consider $f$ restricted to $\fixs$, namely
$$
f|_\fixs:\fixs\to\R\cup\{+\infty\}.
$$
Let us now prove that $f|_\fixs$ admits a nontrivial critical point
$u\in\fixs$, by applying the Mountain Pass theorem for semi-continuous functionals (cf. e.g. \cite[Theorem 3.9]{pelsqu}). 
This result requires condition~\eqref{keycond} to
be satisfied, which holds true by Proposition~\ref{chiavesemicontfunct}.
In light of $j(s,t)\geq \alpha_0 t^p$, there exist $\rho_0>0$ sufficiently small and $\sigma_0>0$ such that
$f(u)\geq \sigma_0$ for any $u\in\fixs$ with $\|u\|_{1,p}=\rho_0$. Moreover,
fixed a nonzero function $\psi\in \fixs\cap L^\infty(\Omega)$, by virtue of assumption \eqref{j6}, we can 
find a positive number $\tau$ such that  
$$
\alpha(\tau\|\psi\|_\infty)\tau^p\leq (2q)^{-1}\|\psi\|_q^q\|D\psi\|^{-p}_p \tau^q
\quad
\text{and}
\quad \tau>\max\Big\{\Big\{\frac{2q}{p}
\frac{\|\psi\|_p^p}{\|\psi\|_q^q}\Big\}^{\frac{1}{q-p}}\!\!\!,\frac{\rho_0}{\|\psi\|_{1,p}}\Big\}.
$$ 
In turn, setting $v_1:=\tau\psi\in \fixs$, we obtain $\|v_1\|_{1,p}>\rho_0$ and 
\begin{equation*}
f(v_1)  \leq \alpha(\tau\|\psi\|_\infty)\tau^p\|D\psi\|^p_p+\tau^p\frac{\|\psi\|^p_p}{p}-\tau^q\frac{\|\psi\|^q_q}{q} 
 \leq \tau^p\frac{\|\psi\|^p_p}{p}-\tau^q\frac{\|\psi\|^q_q}{2q}<0.
\end{equation*}
Therefore $\inf\{f(u):u\in\fixs,\|u\|_{1,p}=\rho_0\}\geq\sigma_0>0=\max\{f(0),f(v_1)\}$, 
so that $f$ admits a positive Mountain Pass value.
Furthermore, for any $c\in\R$, $f$ satisfies the Palais-Smale condition at level $c$. 
Although the proof is essentially contained in~\cite{pelsqu}, for the sake
of self-containedness, we shall provide a sketch in the following highlighting the main differences.
Let then $(u_h)\subset\fixs$ with $f(u_h)\to c$ and $|df|_{\fixs}|(u_h)\to 0$, as $h\to\infty$.
Then, by combining Proposition~\ref{densitypro}, Proposition~\ref{linkkkfix} and the Hahn-Banach theorem, there exists $(w_h)\subset (\fixs)'$
such that $\|w_h\|_{-1,p'}\to 0$ as $h\to\infty$,  $j_t(u_h,|D u_h|)|Du_h|\in L^1(\Omega)$ (by
\eqref{stimaslope}, since $|df|_{\fixs}|(u_h)<+\infty$ yields $|dJ|_\fixs|(u)<+\infty$ in light of
\cite[Theorem 4.13(ii) combined with Theorem 5.1(ii)]{campad}) and
\begin{align}
	\label{loccomp}
\int_{\Omega}j_s(u_h,|Du_h|)v & +\int_{\Omega}
j_t(u_h,|Du_h|)\frac{Du_h}{|Du_h|}\cdot D v   \\
& +\int_\Omega |u_h|^{p-2}u_hv=\int_\Omega |u_h|^{q-2}u_hv+\langle w_h,v\rangle,  \notag
\end{align}
for every $v\in V_{u_h}\cap \fixs$. 
By assumption \eqref{j5}, the sequence $(u_h)\subset\fixs$ is bounded in $\fixs$ 
(see e.g.~\cite[Lemma 3.15]{sq-toulo}). 
Moreover, by means of 
Proposition~\ref{testenlarging}, due to~\eqref{j4} we can choose the test $v=u_n$ in \eqref{loccomp}, yielding
\begin{align*}
\int_{\Omega}j_s(u_h,|Du_h|)u_h+\int_{\Omega}j_t(u_h,|Du_h|)|Du_h|+
\int_\Omega |u_h|^{p}=\int_\Omega |u_h|^{q}+\langle w_h,u_h\rangle.  
\end{align*}
By the sign condition~\eqref{j4} and the boundedness of $(u_h)$ in $W^{1,p}_0(\Omega)$, this yields
\begin{equation}
	\label{supj_t}
\sup_{h\geq 1}\int_{\Omega}j_t(u_h,|Du_h|)|Du_h|<+\infty,\qquad
\sup_{h\geq 1}\int_{\Omega}j_s(u_h,|Du_h|)u_h<+\infty.
\end{equation}
We claim that, up to a subsequence, $(u_h)$ is strongly convergent to some $u\in\fixs$. In this step,
of course, the compact injections~\eqref{immersionicompattt} will play a crucial r\v ole. Let $u\in\fixs$ be the weak limit 
of the sequence $(u_h)$ in $\fixs$. Since $j_t(s,t)t\geq 0$, by Fatou's Lemma the first 
formula in~\eqref{supj_t} yields $j_t(u,|Du|)|Du|\in L^1(\Omega)$.

By following the first part of the proof of~\cite[Theorem 5.1]{pelsqu}, we get $D u_h(x)\to Du(x)$ 
a.e.~in $\Omega$. More precisely, given a sequence of bounded $G_N$-invariant invading domains $(\Omega_k)$ of $\Omega$, one is lead
to the application of \cite[Theorem 5]{dalmur} to a suitable variational identity involving a
class of Leray-Lions operators $b_h(x,\xi)$ related to $j(s,|\xi|)$, $(\tilde w_h)\subset (\fixs)'$ convergent
in $\|\cdot\|_{-1,p'}$ and a sequence $(f_h)$ bounded in
$L^1(\Omega_k)$. This is achieved by testing \eqref{loccomp}
with $H(u_h/k)v\in V_{u_h}\cap \fixs$, where $v\in W^{1,p}_0(\Omega_k)$ is any bounded 
$G_N$-invariant function. In fact, after suitably rearranging the terms arising in the calculation, one is lead to
\begin{equation*}
	\begin{cases}
		\displaystyle
\int_{\Omega_k} b_h(x,Du_h)\cdot Dv=\langle \tilde w_h,v\rangle+\int_{\Omega_k} f_hv,  &\\
\noalign{\vskip3pt}
\,\,\text{for every $G_N$-invariant function $v\in W^{1,p}_0(\Omega_k)\cap L^\infty(\Omega_k)$}. & 
\end{cases}
\end{equation*}
We now point out that, in our framework, differently from~\cite{dalmur} 
(cf.~\cite[Formula 19]{dalmur}) the above variational identity only holds  
for $G_N$-invariant test functions. On the other hand, since $\Omega_k$ and $u_h$ are $G_N$-invariant, 
this is actually sufficient to succeed. In fact, for a suitable sequence of numbers $\delta_k>0$, 
the only test functions which arise in the proof of 
\cite[Theorem 5]{dalmur} are indeed $G_N$-invariant functions, being exactly of the form 
$$
v_h:=\varphi_K\,\psi_{\delta_h}\circ (u_h-u)\in W^{1,p}_0(\Omega_k)\cap L^\infty(\Omega_k),
$$ 
where $K$ is a fixed $G_N$-invariant compact subset of $\Omega_k$,
$\varphi_K\in C^\infty_c(\Omega_k)$ with $0\leq\varphi_K\leq 1$ is a $G_N$-invariant function with $\varphi_K=1$ on $K$ and 
${\rm supt}(\varphi_K)\subseteq U\subseteq\bar U\subseteq \Omega_k$
for some $G_N$-invariant open set $U\subseteq\Omega_k$ and $\psi_\delta\in C^1_c(\R)$, 
$\psi_\delta(s)=s$ for $\{|s|\leq \delta\}$, $\psi_\delta(s)=0$ for $\{|s|\geq 2\delta\}$,
$|\psi_\delta|\leq 2\delta$ and $|D\psi_\delta|$ is bounded. Due to the $G_N$-invariance of $\Omega_k$
restricting to $G_N$-invariant compacts $K$ of $\Omega_k$ is without loss of generality.
In conclusion, \cite[Theorem 5]{dalmur} allows to conclude that
$D u_h(x)\to Du(x)$ for a.e.~$x$ in the set $E_k=\{x\in\Omega_k:|u(x)|\leq k\}$.
Finally, the property on $\Omega$ follows by a standard diagonal argument on the $E_k$s.
Moreover, taking suitable test functions in $\fixs$
(cf. \cite[test (5.8) and the test above (5.11)]{pelsqu}), we get
\begin{equation}
	\label{loccomp-2}
\int_{\Omega}j_s(u,|Du|)v+
\int_{\Omega} j_t(u,|Du|)\frac{Du}{|Du|}\cdot Dv+\int_\Omega |u|^{p-2}uv=\int_\Omega |u|^{q-2}uv,
\end{equation}
for every $v\in V_u\cap \fixs$. For $M>0$, consider the Lipschitz function $\zeta:\R\to\R^+$, 
\begin{equation}
	\label{zetadef}
\zeta (s)=M|s|\quad\text{for $|s|\leq R$},\qquad
\zeta (s)=MR\quad\text{for $|s|\geq R$}.
\end{equation}
By~\eqref{j4}, the growth condition on $j_s$ and $j_t(s,t)t\geq \alpha_0 t^p$, we can easily choose $M$ (depending
upon $R$ and $\alpha_0$) such that
$[j_s(s,t)+j_t(s,t)t\zeta'(s)]s\geq 0$ for all $s\in\R$ and $t\in\R^+$.
Recalling that $j_t(u_h,|Du_h|)|Du_h|\in L^1(\Omega)$ and $j_t(u,|Du|)|Du|\in L^1(\Omega)$,
by Proposition~\ref{testenlarging}, condition \eqref{j4} and the definition of $\zeta$, 
the function $v_h:=u_h e^{\zeta(u_h)}\in\fixs$ (resp. $v:=u e^{\zeta(u)}\in\fixs$)
can be taken as admissible test functions into \eqref{loccomp} (resp. in \eqref{loccomp-2}).
Finally, since by virtue of Lemma~\ref{compsobemb} $(u_h)$ strongly converges to $u$ in $L^q(\Omega)$, 
by dominated convergence,
$$
\lim_h\int_\Omega |u_h|^{q}e^{\zeta(u_h)}=\int_\Omega |u|^{q}e^{\zeta(u)}.
$$
Therefore, using Fatou's Lemma twice, we reach
\begin{align*}
&	\int_\Omega j_t(u,|Du|)|Du|e^{\zeta(u)}+\int_\Omega |u|^{p}e^{\zeta(u)} \\
& \leq \liminf_h
\Big[\int_\Omega j_t(u_h,|Du_h|)|Du_h| e^{\zeta(u_h)}+\int_\Omega |u_h|^{p}e^{\zeta(u_h)}\Big] \\
& \leq \limsup_h
\Big[\int_\Omega j_t(u_h,|Du_h|)|Du_h| e^{\zeta(u_h)}+\int_\Omega |u_h|^{p}e^{\zeta(u_h)}\Big] \\
& = \int_\Omega |u|^{q}e^{\zeta(u)}   
 -\liminf_h \int_\Omega \left[ j_s(u_h,|Du_h|)+
j_t(u_h,|D u_h|)|Du_h|\zeta'(u_h)\right] u_h e^{\zeta(u_h)} \\
& \leq \int_\Omega |u|^{q}e^{\zeta(u)}  
 -\int_\Omega \left[ j_s(u,Du)+
j_t(u,|Du|)|Du|\zeta'(u)\right] u e^{\zeta(u)} \\
& =\int_\Omega j_t(u,|Du|)|Du|e^{\zeta(u)}+\int_\Omega |u|^{p}e^{\zeta(u)},
\end{align*}
which combined with the pointwise convergence of the gradients and the inequalities
\begin{align*}
\alpha_0 |Du_h|^p  & \leq j_t(u_h,|Du_h|)|D u_h| e^{\zeta(u_h)}+|u_h|^{p}e^{\zeta(u_h)} \\
|u_h|^p & \leq j_t(u_h,|Du_h|)|Du_h| e^{\zeta(u_h)}+|u_h|^{p}e^{\zeta(u_h)}
\end{align*}
implies that $\|u_h\|_p\to \|u\|_p$  and $\|Du_h\|_p\to \|D u\|_p$ as $h\to\infty$,
by dominated convergence. By the uniform
convexity of $L^p$, we conclude $u_h\to u$ in $W^{1,p}_0(\Omega)$ as $h\to\infty$. 
The Palais-Smale condition is thus verified.
By applying \cite[Theorem 3.9]{pelsqu}, we find a nontrivial 
Mountain Pass critical point $u\in \fixs$, namely $|df|_\fixs|(u)=0$. In turn,
in light of Proposition~\ref{connection}, $0\in \partial f|_\fixs(u)$. 
Hence, by Theorem~\ref{abstr22} (more precisely by Corollary~\ref{corconcrabs2}), we get
$$
\forall v\in V_u:\quad f'(u)v=0.
$$
Now, by Proposition~\ref{derivate}, we have
$$
\forall v\in V_u:\quad \int_{\Omega}j_t(u,|Du|)\frac{Du}{|Du|}\cdot Dv
+\int_{\Omega}j_s(u,|Du|)v+\int_{\Omega}|u|^{p-2}uv=\int_{\Omega}|u|^{q-2}uv.
$$
In light of Proposition~\ref{linkkkfix}, it follows 
$j_t(u,|Du|)|Du|\in L^1(\Omega)$ and $j_s(u,|Du|)u\in L^1(\Omega)$.
Then, it holds $|j_s(u,|Du|)|\leq |j_s(u,|Du|)|\chi_{\{|u|\leq 1\}}+|j_s(u,|Du|)u|$. 
Also, for a fixed compact $K\subset\Omega$, it holds
\begin{align*}
|j_t(u,|Du|)|\chi_K & \leq |j_t(u,|Du|)\chi_{\{|Du|\geq 1\}}\chi_K|+|j_t(u,|Du|)\chi_{\{|Du|\leq 1\}}\chi_K| \\
& \leq j_t(u,|Du|)|Du|+C\alpha(|u|)\chi_K \\
& \leq j_t(u,|Du|)|Du|+C|u|^{p^*}+C\chi_K.
\end{align*}
We have thus proved that $j_s(u,|Du|)\in L^1(\Omega)$ and $j_t(u,|Du|)\in L^1_{{\rm loc}}(\Omega)$. 
In turn, for any $v\in C^\infty_c(\Omega)$, we have $j_s(u,|Du|)v\in L^1(\Omega)$ and
$j_t(u,|Du|)\frac{Du}{|Du|}\cdot Dv\in L^1(\Omega)$, so that, in light of \cite[Theorem 4.8]{pelsqu} 
(which extends to the current setting), we get
$$
\forall v\in C^\infty_c(\Omega):\, \int_{\Omega}j_t(u,|Du|)\frac{Du}{|Du|}\cdot Dv
+\int_{\Omega}j_s(u,|Du|)v+\int_{\Omega}|u|^{p-2}uv=\int_{\Omega}|u|^{q-2}uv,
$$
namely $u$ is a distributional solution. The proof is now complete. \qed
\vskip5pt
\noindent

\subsubsection{Proof of Corollary~\ref{appl-cor1}}
Let $\Omega=\R^N$ and take $X=W^{1,p}(\R^N)$ and $G={\mathscr O}(N)$. Then, the assertion follows directly 
from Theorem~\ref{appl} since the invariance $u(gx)=u(x)$ for all $g\in G$ 
is equivalent to the radial symmetry of $u$. To obtain that $u$ is nonnegative, in the application of Theorem~\ref{appl} it suffices to replace,
in the definition of $f$, the term $\int_{\R^N}|u|^{q}/q$ with $\int_{\R^N}G(u)$, where $G(s)=(s^+)^{q}/q$. Hence, 
as in the proof of Theorem~\ref{appl}, after application of our second Palais symmetric criticality
principle, one finds a $u\in {\rm Fix}(G)$ which solves
$$
\int_{\R^N}j_t(u,|Du|)\frac{Du}{|Du|}\cdot Dv
+\int_{\R^N}j_s(u,|Du|)v+\int_{\R^N}|u|^{p-2}uv=\int_{\R^N}|u^+|^{q-2}uv,
$$
for all $v\in V_u$. If $\zeta$ is the map defined in~\eqref{zetadef}, 
the function $\tilde v=-u^- e^{\zeta(u)}$ can be taken as an admissible test
function in the above identity by \cite[Theorem 4.8]{pelsqu} in view of the growth assumptions on $j_s,j_t$ and
of the sign condition~\eqref{j4}, yielding
\begin{align*}
0 &=\int_{\R^N}|u^+|^{q-2}u\tilde v=\int_{\R^N}e^{\zeta(u)} j_t(u,|Du^-|)|Du^-| \\
& +\int_{\R^N}-[j_s(u,|Du|)+j_t(u,|Du|)|Du|\zeta'(u)]u^- e^{\zeta(u)}+\int_{\R^N}|u^-|^{p}e^{\zeta(u)} \\
&\geq \int_{\R^N}|u^-|^{p}e^{\zeta(u)}\geq \int_{\R^N}|u^-|^p\geq 0,
\end{align*}
yielding $u^-=0$, hence $u\geq 0$. This concludes the proof. \qed

\vskip25pt
\noindent
{\bf Acknowledgment.} The author thanks Marco Degiovanni for some very useful discussions.

\medskip
\bigskip

\end{document}